\documentclass[11pt,reqno]{amsart}
\usepackage{graphicx}
\usepackage{float}
\usepackage{amsfonts}
\usepackage[caption = false]{subfig}
\usepackage{amsmath}
\usepackage{graphics}
\usepackage{float}
\usepackage{enumerate,comment}
\usepackage{mathtools}
\usepackage{amsthm}
\usepackage[english]{babel}
\usepackage{amsfonts,amssymb}% insert here the call for the packages your document requires
\usepackage{mathrsfs}
\usepackage[utf8x]{inputenc}\usepackage[english]{babel}
\usepackage{soul}
\usepackage[all]{xy,xypic}
\usepackage{cite}
\usepackage{relsize}
\numberwithin{equation}{section}
\newtheorem{theorem}{Theorem}[section]

\newtheorem{lemma}[theorem]{Lemma}

\theoremstyle{definition}

\theoremstyle{remark}

\numberwithin{equation}{section}
\DeclareMathOperator{\RE}{Re}

\begin{document}
	%\fontsize{14pt}{16pt}\selectfont
	\title[On Hankel Determinants related with Bean shaped domain]{On estimation of Hankel determinants for certain class of starlike functions}

	\author[S. S. Kumar]{S. Sivaprasad Kumar}
	\address{Department of Applied Mathematics, Delhi Technological University, Delhi--110042, India}
	\email{spkumar@dce.ac.in}
	
	\author[N. Verma]{Neha Verma}
	\address{Department of Applied Mathematics, Delhi Technological University, Delhi--110042, India}
	\email{nehaverma1480@gmail.com}

	
	\maketitle
\begin{abstract}
In the present study, we consider two subclasses starlike and convex functions, denoted by $\mathcal{S}_{\mathcal{B}}^{*}$ and $\mathcal{C}_{\mathcal{B}}$ respectively, associated with a bean-shaped domain. Further, we estimate certain sharp initial coefficients, as well as second, third and fourth-order Hankel determinants for functions belonging to the class $\mathcal{S}_{\mathcal{B}}^{*}$. Additionally, we compute sharp second and third-order Hankel determinants for functions belonging to the $\mathcal{C}_{\mathcal{B}}$ class.

\end{abstract}
\maketitle
	
\section{Introduction}
	\label{intro}
\noindent Let $\mathcal{A}$ denote the class of normalized analytic functions defined on the open unit disk $\mathbb{D}:=\{z\in \mathbb{C}:|z|<1\}$, having the form
\begin{equation}
f(z) = z + \sum_{n=2}^{\infty}a_nz^n\label{form}
\end{equation}
and suppose $\mathcal{S}$ be a subclass of $\mathcal{A}$ comprising univalent functions. Consider $\mathcal{P}$ to be the class of analytic functions defined on $\mathbb{D}$ with a positive real part, expressed as $p(z)=1+\sum_{n=1}^{\infty}p_n z^n$. Suppose $h$ and $g$ are two analytic functions, we say $h$ is subordinate to $g$, symbolically $h\prec g$, if there exists a Schwarz function $w$ with $w(0)=0$ and $|w(z)|\leq |z|$ such that $h(z)=g(w(z))$. A substantial body of literature exists on coefficient problems, ranging from the seminal Bieberbach's conjecture of 1916 to contemporary research (see \cite{goodman vol1}). There are two prominent subclasses of $\mathcal{S}$ consisting of starlike and convex functions, respectively denoted by $\mathcal{S}^{*}$ and $\mathcal{C}$. Further, in 1992, Ma and Minda \cite {ma-minda} unfified various subclasses of $\mathcal{S}^{*}$ and $\mathcal{C}$ by introducing the following two classes:
\begin{equation}
		\mathcal{S}^{*}(\varphi)=\bigg\{f\in \mathcal {A}:\dfrac{zf'(z)}{f(z)}\prec \varphi(z) \bigg\}\label{mindaclass}
\end{equation}
and 
\begin{equation}
		\mathcal{C}(\varphi)=\bigg\{f\in \mathcal {A}:1+\dfrac{zf''(z)}{f'(z)}\prec \varphi(z) \bigg\}\label{mindaclassc}
\end{equation}
where $\varphi$ is an analytic univalent function satisfying the conditions $\RE\varphi(z)>0$, $\varphi(\mathbb{D})$ symmetric about the real axis and starlike with respect to $\varphi(0)=1$ with $\varphi'(0)>0$. Recently, many Ma-Minda classes are introduced and studied by several authors by appropriately choosing $\varphi(z)$ in \eqref{mindaclass}. See the various Ma-Minda subclasses of starlike functions listed in the first column with the corresponding choice of $\varphi(z)$ in second column of Table \ref{7 table1}.

The concept of Hankel determinants, introduced in 1966 (see \cite{pomi}), continues to be a topic of significant interest for researchers today. The definition of the $q{th}$ Hankel determinant $H_{q}(n)$ of analytic functions $f \in \mathcal{A}$, under the assumption that $a_1:=1$, is as follows:
\begin{equation}
		H_{q}(n) =\begin{vmatrix}
			a_n&a_{n+1}& \ldots &a_{n+q-1}\\
			a_{n+1}&a_{n+2}&\ldots &a_{n+q}\\
			\vdots& \vdots &\ddots &\vdots\\
			a_{n+q-1}&a_{n+q}&\ldots &a_{n+2q-2}\label{5hqn}
		\end{vmatrix}, \quad n,q\in \mathbb{N}.
\end{equation}
The expressions for the second and third-order Hankel determinants for specific values of $q$ and $n$, are denoted by $H_{2}(3)$ and $H_{3}(1)$, respectively, given by
\begin{equation}
    H_{2}(3):=a_3a_5-a_4^2,\label{3h23}
\end{equation}
and
\begin{equation}
H_{3}(1):=2 a_2a_3a_4-a_3^3-a_4^2-a_2^2a_5+a_3a_5.\label{1h3}
\end{equation}
\begin{table}\label{7 table1}
\caption{\centering List of sharp third order Hankel determinants}
\centering
\begin{tabular}{|c|c| c |c|} 
 \hline
 Class& $\varphi(z)$ & Sharp $|H_3(1)|$ & Reference \\ [1ex] 
 \hline
  $\mathcal{S}^{*}$ & $(1+z)/(1-z)$  & $4/9$ &\cite{banga,4/9}    \\
 %$\mathcal{S}^{*}(1/2)$&   1/9 &\cite{rath}   \\
 $\mathcal{S}^{*}_{\varrho}$ & $1+ze^z$ & 1/9 &\cite{nehacardioid}\\
  $\mathcal{SL}^{*}$ & $\sqrt{1+z}$ & 1/36 &\cite{sharp}\\
  $\mathcal{S}^{*}_{e}$ & $e^z$ & 1/9&\cite{nehaexpo} \\
  $\mathcal{S}^{*}_{\rho}$& $1+\sinh^{-1}(z)$ & 1/9&\cite{nehapetal} \\%[1ex]
  $\mathcal{S}^{*}_{\tau}$& $1+\arctan z$ & 1/9&\cite{nehastrip} \\%[1ex]
 \hline
\end{tabular}
\end{table}
\noindent Deriving a sharp bound for Hankel determinants is a formidable challenge, prompting numerous researchers to endeavor to do so for various subclasses of starlike functions, see \cite{alarif,sharp,bulboca,kanas1,bulboca2} and some are listed in third column of Table \ref{7 table1}. 

Recently, Kumar and Yadav \cite{poojabean}, by choosing $\varphi(z)=\sqrt{1+\tanh z}$, introduced and studied the Ma-Minda subclass of starlike functions $\mathcal{S}_{\mathcal{B}}^{*}$ associated with a bean-shaped domain, given by
\begin{equation*}
    \mathcal{S}_{\mathcal{B}}^{*}=\bigg\{f\in \mathcal {A}:\dfrac{zf'(z)}{f(z)}\prec \sqrt{1+\tanh z}\bigg\}.
\end{equation*}
Motivated by it, we introduce the following convex counterpart of the above class:
\begin{equation*}
   \mathcal{C}_{\mathcal{B}}=\bigg\{f\in \mathcal {A}:1+\dfrac{zf''(z)}{f'(z)}\prec \sqrt{1+\tanh z}\bigg\}. 
\end{equation*}
The authors have investigated the geometric properties of the univalent function $\sqrt{1+\tanh z}$, along with some inclusion and sharp radius results involving $\mathcal{S}_{\mathcal{B}}^{*}$, as well as implications of first-order differential subordination. This class can be further studied to know more about the behavior of the coefficients of functions belonging to this class, second and third order differential subordination,  etc. Thus, studying such subclasses of starlike and convex functions opens new avenues in the field of research. Taking this aspect in account, in our current investigation, we focus on coefficient-related problems concerning the aforementioned classes $\mathcal{S}_{\mathcal{B}}^{*}$ and $\mathcal{C}_{\mathcal{B}}$, which are yet not addressed in the literature. In fact, we are finding the sharp bounds of the initial coefficients, second and third-order Hankel determinants, as well as the bound of the fourth-order Hankel determinant for functions belonging to the class $\mathcal {S}^{*}_{\mathcal{B}}$. Additionally, we establish sharp bounds of the second and third-order Hankel determinants for functions belonging to the class $\mathcal{C}_{\mathcal{B}}$.

\section{Coefficient related problems for $\mathcal{S}^{*}_{\mathcal{B}}$}
In this section, we start by determining the sharp bounds of the initial coefficients $a_i$ for $(i=2,3,4,5)$ followed by establishing the sharp bounds of the second and third-order Hankel determinants for functions $f\in {S}^{*}_{\mathcal{B}}$. Subsequently, we derive the bounds for 
$a_6$ and $a_7$ to deduce a possible bound of the fourth-order Hankel determinant for functions $f\in {S}^{*}_{\mathcal{B}}$.
	
\subsection{Sharp initial coefficient bounds}

Let $f\in \mathcal{S}_{\mathcal{B}}^{*},$ then there exists a Schwarz function $w(z)$ such that
\begin{equation}\label{7 formulaai}
	\dfrac{zf'(z)}{f(z)}=\sqrt{1+\tanh w(z)}.
\end{equation}
Suppose that $p(z)=1+p_1z+p_2z^2+\cdots\in \mathcal{P}$ and consider $w(z)=(p(z)-1)/(p(z)+1)$. Further, by substituting the expansions of $w(z)$, $p(z)$ and $f(z)$ in equation (\ref{7 formulaai}) and then comparing the coefficients, we obtain the expressions of $a_i (i=2,3,...,7)$ in terms of $p_j (j=1,2,...,5)$, given as
\begin{equation}\label{7b sa2}
a_2=\dfrac{p_1}{4},\quad a_3=\dfrac{1}{64}\bigg(8p_2-3p_1^2\bigg), \quad a_4=\dfrac{1}{2304}\bigg(23p_1^3 -168p_1 p_2 + 192 p_3\bigg),
\end{equation}
\begin{equation}\label{7b sa5}
a_5=\dfrac{1}{18432}\bigg( -11p_1^4+528 p_1^2 p_2-576p_2^2 -1056 p_1 p_3+1152 p_4\bigg),
\end{equation}
 \begin{equation}\label{7b sa6}
a_6=\dfrac{1}{1843200}\bigg( 50880 p_1 p_2^2-2367 p_1^5 - 8560 p_1^3 p_2   + 46080 p_1^2 p_3-96000 p_2 p_3 - 86400 p_1 p_4 + 92160 p_5\bigg)
\end{equation}
and 
\begin{align}\label{7b sa7}
a_7&=\dfrac{1}{530841600}\bigg( 601421 p_1^6 - 2365320 p_1^4 p_2 - 4818240 p_1^2 p_2^2 + 4723200 p_2^3 +26150400 p_1 p_2 p_3\nonumber\\
 &\quad\quad\quad\quad\quad - 
 2648640 p_1^3 p_3 - 11980800 p_3^2 + 11577600 p_1^2 p_4 - 
 23500800 p_2 p_4 - 21012480 p_1 p_5\bigg).
\end{align}	
The results stated below are necessary for proving our main result.
\begin{lemma}\cite{rj}
Let $p\in \mathcal{P}$ be of the form $1+\sum_{n=1}^{\infty}p_nz^n.$ Then
\begin{equation}
        |p_1^4-3p_1^2p_2+p_2^2+2p_1p_3-p_4|\leq 2\label{p}
\end{equation}
and
\begin{equation}
       |p_3-2p_1p_2+p_1^3|\leq 2.\label{q}
    \end{equation}
    \end{lemma}
\begin{lemma}\cite{ma-minda}
Let $p\in \mathcal{P}$ be of the form $1+\sum_{n=1}^{\infty}p_nz^n.$ Then
\begin{align*}
      |p_2-\beta p_1^2|\leq \begin{cases} 2-4\beta,& \beta\leq 0;\\
      2, & 0\leq \beta\leq 1;\\
      4\beta-2, & \beta\geq 1
      \end{cases}
\end{align*}
when $\beta<0$ or $\beta>1,$ the equality holds if and only if $p(z)=(1+z)/(1-z)$ or one of its rotations. If $0<\beta<1,$ then the inequality holds if and only if $p(z)=(1+z^2)/(1-z^2)$ or one of its rotations. If $\beta=0,$ the equality holds if and only if $p(z)=(1+\eta)(1+z)/(2(1-z))+(1-\eta)(1-z)/(2(1+z))(0\leq \eta\leq 1)$ or one of its rotations. If $\beta=1,$ the equality holds if and only if p is the reciprocal of one of the functions such that the equality holds in case of $\beta=0.$ Though the above upper bound is sharp for $0<\beta<1,$ still it can be improved as follows:
\begin{equation}
      |p_2-\beta p_1^2|+\beta|p_1|^2\leq 2 \quad (0<\beta\leq 1/2)\label{use}
\end{equation}
and 
\begin{equation*}
      |p_2-\beta p_1^2|+(1-\beta)|p_1|^2\leq 2 \quad (1/2<\beta\leq 1).
\end{equation*}
\end{lemma}
Also, we recall that 
\begin{equation}
    \max_{0\leq t\leq 4}(At^2+Bt+C)=
    \begin{cases}
    C,& B\leq0, A\leq \frac{-B}{4};\\
    16A+4B+C, & B\geq0, A\geq \frac{-B}{8}\quad \text{or}\quad B\leq 0, A\geq \frac{-B}{4};\\
    \dfrac{4AC-B^2}{4A},& B>0, A\leq \frac{-B}{8}.\label{l}
    \end{cases}
\end{equation}

\begin{theorem}\label{7 initialbound}
If $f\in\mathcal{S^{*}_{\mathcal{B}}}$, then $(i)$ $|a_2|\leq1/2$, $(ii)$ $|a_3|\leq1/4$, $(iii)$ $|a_4|\leq1/6$ and $(iv)$ $|a_5|\leq847/3216\simeq 0.263371\cdots$. These bounds are sharp.
\end{theorem}
\begin{proof}
 $(i)$ Since $|p_n|\leq2$ for $n\geq1$, therefore, from \eqref{7b sa2}, $|a_2|\leq1/2.$\\
$(ii)$ For $a_3$, we use the inequality $|p_2-\mu p_1^2| \leq 2\max\{1, |2\mu-1|\}$ given by Ma and
Minda \cite{ma-minda}, which yields $|a_3| \leq 1/4.$\\
$(iii)$ For $a_4$, (\ref{7 formulaai}) is re-written as:
\begin{equation}
		zf'(z)=\sqrt{1+\tanh (w(z))}f(z). \label{7 re}
\end{equation}
On substituting $f(z)=z+\sum_{n=2}^{n=\infty}a_nz^n$ and $w(z)=\sum_{k=1}^{\infty}w_{k}z^k$ in (\ref{7 re}) and comparing the coefficients of $z^4$, we get $$6a_4=\left(w_3+\dfrac{1}{4}w_1w_2-\dfrac{13}{48}w_1^3\right).$$
Now using \cite[Lemma 2, p.128]{poko}, we deduce that $|6a_4|\leq1$ and hence the result follows.\\
$(iv)$ From \eqref{7b sa5}, we get the expression of $a_5$ as
\begin{align*}
    a_5&=\dfrac{1}{16}\bigg(-\dfrac{11}{1152}p_1^4+\dfrac{11}{24}p_1^2p_2-\dfrac{p_2^2}{2}-\dfrac{11}{12}p_1p_3+p_4\bigg)\\
    &=\dfrac{1}{16}\bigg(-\dfrac{1}{2}P+\dfrac{1}{12}p_1Q-\dfrac{7}{8}p_1^2R+\dfrac{1}{2}p_4\bigg),
   \end{align*}
which further gives
\begin{equation*}
    |a_5|\leq \dfrac{1}{16}\left(\dfrac{1}{2}|P|+\dfrac{1}{12}|p_1||Q|+\dfrac{7}{8}|p_1|^2|R|+\dfrac{1}{2}|p_4|\right),
\end{equation*}
where $P=p_1^4-3p_1^2p_2+p_2^2+2p_1p_3-p_4,$ $Q=p_3-2p_1p_2+p_1^3$ and $R=p_2-(67/144)p_1^2$. Moreover, using the bounds of $|P|\leq 2$ from (\ref{p}), $|Q|\leq 2$ from  (\ref{q}) and $|R|\leq 2$ from (\ref{use}), respectively, we obtain
   $$|a_5|\leq\dfrac{1}{16}\left(\dfrac{7}{3}+\dfrac{7}{4}|p_1|^2-\dfrac{469}{1152}|p_1|^4\right).$$ 
Now, we obtain $|7|p_1|^2/4-469|p_1|^4/1152|\leq 126/67$ using (\ref{l}) by taking $A=-469/1152$, $B=7/4$ and $C=0,$ which leads to the desired estimate for $|a_5|$.  The sharpness of the result can be witnessed when $p_1=p_4=2$, $p_2=(737+\sqrt{963326})/402$ and $p_3=-2$.  
   The function 
	\begin{equation*}
	 f_n(z)=z\exp\bigg(\int_{0}^{z}\frac{\sqrt{1+\tanh (t^{n-1})}-1}{t}dt\bigg)  
	\end{equation*}
acts as the extremal function for the initial coefficients $a_n$ for $n=2,3$ and $4$.
\end{proof}
The formula for $p_j$ $(j=2,3,4)$, which is included in the Lemma \ref{pformula} below, plays a vital role in establishing the sharp bounds for Hankel determinants and forms the foundation for our main results. 
\begin{lemma}\cite{rj,lemma1}\label{pformula}
Let $p\in \mathcal {P}$ has the form $1+\sum_{n=1}^{\infty}p_n z^n.$ Then for some $\gamma$, $\eta$ and $\rho$ such that $|\gamma|\leq 1$, $|\eta|\leq 1$ and $|\rho|\leq 1$, we have
\begin{equation}
2p_2=p_1^2+\gamma (4-p_1^2),\label{b2}
\end{equation}
\begin{equation}
4p_3=p_1^3+2p_1(4-p_1^2)\gamma -p_1(4-p_1^2) {\gamma}^2+2(4-p_1^2)(1-|\gamma|^2)\eta, \label{b3}
\end{equation}
and \begin{equation}
8p_4=p_1^4+(4-p_1^2)\gamma (p_1^2({\gamma}^2-3\gamma+3)+4\gamma)-4(4-p_1^2)(1-|\gamma|^2)(p_1(\gamma-1)\eta+\bar{\gamma}{\eta}^2-(1-|\eta|^2)\rho). \label{b4}
\end{equation}

\end{lemma}

\subsection{Sharp Hankel determinants for $\mathcal{S}^{*}_{\mathcal{B}}$}	

The following theorem presents the sharp bound for $|H_{3}(1)|$ for functions belonging to the class $\mathcal{S}^{*}_{\mathcal{B}}$.

\begin{theorem}\label{7b1 ssharph31}
Let $f\in \mathcal {S}^{*}_{\mathcal{B}}$, then 
\begin{equation}
|H_{3}(1)|\leq 1/36.\label{7b1 s9.5}
\end{equation}
This result is sharp.
\end{theorem}

\begin{proof}  
Since the class $\mathcal {P}$ is invariant under rotation, we can take $p_1=:p$ belongs to the interval [0,2]. Substitute the values of $a_i(i=2,3,4,5)$ in (\ref{1h3}) from (\ref{7b sa2}) and (\ref{7b sa5}). We get
\begin{align*}
H_{3}(1)&=\dfrac{1}{21233664}\bigg(-3511 p^6 - 5160 p^4 p_2 - 14400 p^2 p_2^2 - 124416 p_2^3 + 56256 p^3 p_3\\
&\quad \quad\quad \quad \quad \quad +216576 p p_2 p_3 - 147456 p_3^2 - 145152 p^2 p_4 +165888p_2 p_4\bigg).
\end{align*}
After simplifying the calculations using (\ref{b2})-(\ref{b4}), we obtain				
$$H_{3}(1)=\dfrac{1}{21233664}\bigg(\beta_1(p,\gamma)+\beta_2(p,\gamma)\eta+\beta_3(p,\gamma){\eta}^2+\beta_4(p,\gamma,\eta)\rho\bigg),$$
for $\gamma,\eta,\rho\in \mathbb {D}.$ Here
\begin{align*}
\beta_1(p,\gamma):&=-1099p^6-1872{\gamma}^2p^2(4-p^2)^2-20736\gamma^3(4-p^2)^2-5760{\gamma}^3p^2(4-p^2)^2+1152{\gamma}^4p^2(4-p^2)^2\\
&\quad+3084{\gamma}p^4(4-p^2)+624p^4{\gamma}^2(4-p^2)-7776p^4{\gamma}^3(4-p^2)-31104{\gamma}^2p^2(4-p^2),\\
\beta_2(p,\gamma):&=96(1-|\gamma|^2)(4-p^2)(149p^3+324{\gamma}p^3+228p\gamma(4-p^2)-48p{\gamma}^2(4-p^2)),\\
\beta_3(p,\gamma):&=1152(1-|\gamma|^2)(4-p^2)(-32(4-p^2)-4|\gamma|^2(4-p^2)+27p^2\bar{\gamma}),\\
\beta_4(p,\gamma,\eta):&=10368(1-|\gamma|^2)(4-p^2)(1-|\eta|^2)(4(4-p^2)\gamma-3p^2).
\end{align*}
By choosing $x=|\gamma|$, $y=|\eta|$ and utilizing the fact that $|\rho|\leq 1,$ the above expression reduces to the following:
\begin{align*}
|H_{3}(1)|\leq \dfrac{1}{21233664}\bigg(|\beta_1(p,\gamma)|+|\beta_2(p,\gamma)|y+|\beta_3(p,\gamma)|y^2+|\beta_4(p,\gamma,\eta)|\bigg)\leq A(p,x,y),
\end{align*}
where 
\begin{equation}
A(p,x,y)=\dfrac{1}{21233664}\bigg(a_1(p,x)+a_2(p,x)y+a_3(p,x)y^2+a_4(p,x)(1-y^2)\bigg),\label{7b1 snew}
\end{equation} 
with 
\begin{align*}
a_1(p,x):&=1099p^6+1872x^2p^2(4-p^2)^2+20736x^3(4-p^2)^2+5760x^3p^2(4-p^2)^2+1152x^4p^2(4-p^2)^2\\
&\quad +3084xp^4(4-p^2)+624p^4x^2(4-p^2)+7776p^4x^3(4-p^2)+31104x^2p^2(4-p^2),\\
a_2(p,x):&=96(1-x^2)(4-p^2)(149p^3+324xp^3+228px(4-p^2)+48px^2(4-p^2)),\\
a_3(p,x):&=1152(1-x^2)(4-p^2)(32(4-p^2)+4x^2(4-p^2)+27p^2x),\\
a_4(p,x):&=10368(1-x^2)(4-p^2)(4x(4-p^2)+3p^2).
\end{align*}				
In the closed cuboid $Q:[0,2]\times [0,1]\times [0,1]$, we now maximise $A(p,x,y)$, by locating the maximum values in the interior of the six faces, on the twelve edges, and in the interior of $Q$.
\begin{enumerate}
\item We start by taking into account every internal point of $Q$. Assume that $(p,x,y)\in (0,2)\times (0,1)\times (0,1)$. We calculate $\partial{A}/\partial y$, partially differentiate (\ref{7b1 snew}) with respect to $y$ 
 to identify the points of maxima in the interior of $Q$. We get
\begin{align*}
\dfrac{\partial A}{\partial y}&=\dfrac{(4 - p^2)(1 - x^2)}{21233664}  \bigg(24 p x (5 + 2 x) + p^3 (17 + 24 x - 12 x^2)+96 (8 - 9 x + x^2) y \\
&\quad \quad\quad \quad\quad\quad\quad\quad\quad- 12 p^2 (25 - 27 x + 2 x^2) y\bigg). 
\end{align*}
Now $\partial A/\partial y=0$ gives
\begin{equation*}
y=y_0:=\dfrac{48 p x (19 + 4 x) + p^3 (149 + 96 x - 48 x^2)}{24 (1-x) (16 (x-8) + p^2 (59-4 x))}.
\end{equation*}
The existence of critical points requires that $y_0$ belong to $(0,1)$, which is only possible when 						
\begin{align}
300p^2+864x+24p^2x^2&>17p^3+120px+24p^3x+48px^2-12p^3x^2\nonumber\\
&\quad+768+864x+24p^2x^2.\label{7b1 h1}
\end{align}		
Now, we find the solution satisfying the inequality (\ref{7b1 h1})
for the existence of critical points using the hit and trial method. If we assume $p$ tends to $0$, then there does not exist any $x\in (0,1)$ satisfying (\ref{7b1 h1}). But, when $p$ tends to 2, (\ref{7b1 h1}) holds for all $x<175/648.$ We also observe that there does not exist any $p\in (0,2)$ when $x\in (174/648,1)$.
Similarly, if we assume $x$ tends to 0, then for all $p>1.61687$, (\ref{7b1 h1}) holds. After calculations, we observe that there does not exist any $x\in (0,1)$ when $p\in (0,1.61687).$
Thus, the domain for the solution of the equation is $(1.61687,2)\times (0,175/648).$ Now, we examine that $\frac{\partial A}{\partial y}|_{y=y_0}\neq 0$ in $(1.61687,2)\times (0,175/648).$
So, we conclude that the function $A$ has no critical point in $(0,2)\times (0,1)\times (0,1).$

\item The interior of each of the cuboid $Q$'s six faces is now being considered.\\
\underline{On $p=0$}
\begin{equation}
k_1(x,y):=\dfrac{16y^2-14x^2y^2-2x^4y^2-9x^3(1-2y^2)+18x(1-y^2)}{576},\quad x,y\in (0,1).\label{7b1 9.4}
\end{equation}
 Since
\begin{equation*}
\dfrac{\partial k_1}{\partial y}=\dfrac{(1 - x^2)(x+1)(8-x)y}{144}\neq 0,\quad x,y\in (0,1),
\end{equation*}
indicates that $k_1$ has no critical points in $(0,1)\times(0,1)$.	\\					
\noindent \underline{On $p=2$}
\begin{equation}
A(2,x,y):=\dfrac{1099}{331776}= 0.00331248,\quad x,y\in (0,1).\label{5 9.3}
\end{equation}
\underline{On $x=0$}
\begin{align}
k_2(p,y):&=\dfrac{1099p^6 + (4-p^2)(14304p^3y+36864y^2(4-p^2)+31104p^2(1-y^2)}{21233664}
							\label{7b1 9.1}
\end{align}
with $p\in (0,2)$ and $y\in (0,1).$ To determine the points of maxima, we solve $\partial k_2/\partial p=0$ and $\partial k_2/\partial y=0$. After solving $\partial k_2/\partial y=0,$ we get
\begin{equation}
y=\dfrac{149p^3}{24(59p^2-128)}(=:y_p).\label{7b1 y}
\end{equation}
In order to have $y_p\in (0,1)$ for the given range of $y$, $p_0:=p>\approx 1.61687$ is required. Based on calculations, $\partial k_2/\partial p=0$ gives
\begin{equation}
p(41472 - 20736 p^2 + 1099 p^4 + 28608 p y - 11920 p^3 y - 
 139776 y^2 + 45312 p^2 y^2)=0.\label{7b1 9}
\end{equation}
After substituting (\ref{7b1 y}) in (\ref{7b1 9}), we have
\begin{equation}
75497472p - 107347968 p^3 + 50314752 p^5 - 8246384 p^7 + 133989 p^9=0.\label{7b1 40}
\end{equation}
A numerical calculation suggests that $p\approx 1.24748\in (0,2)$ is the solution of (\ref{7b1 40}). So, we conclude that $k_2$ does not have any critical point in $(0,2)\times(0,1)$.
						
\noindent \underline{On $x=1$}
\begin{equation}
A(p,1,y)=k_3(p,y):=\dfrac{331776 + 99072 p^2 - 34704 p^4 - 1601 p^6}{21233664}, \quad p\in (0,2).\label{7b1 9.2}
\end{equation}
While computing $\partial k_3/\partial p=0$, $p_0:=p\approx 1.14405$ comes out to be the critical point. Undergoing simple calculations, $k_3$ achieves its maximum value $\approx 0.0187629$ at $p_0.$
						
\noindent \underline{On $y=0$} 
\begin{align*}
k_4(p,x):&=\dfrac{1}{21233664}\bigg(-1601 p^6 - 331776 (-1 - 2 x + 2 x^3)+2304 p^2 (97 - 144 x \\
&\quad \quad\quad\quad\quad\quad - 54 x^2 + 144 x^3) -144 p^4 (457 - 288 x - 216 x^2 + 288 x^3)\bigg).
\end{align*}
After further calculations such as,
\begin{align*}
\dfrac{\partial k_4}{\partial x}&=\dfrac{(4-p^2)}{1024}\bigg(8 - 24 x^2 + p^2 (-2 - 3 x + 6 x^2)\bigg)
\end{align*}
and \begin{align*}
\dfrac{\partial k_4}{\partial p}&=\dfrac{1}{3538944}\bigg(-1601 p^5 + 768 p (97 - 144 x - 54 x^2 + 144 x^3) \\
&\quad\quad\quad \quad\quad\quad -96 p^3 (457 - 288 x - 216 x^2 + 288 x^3)  \bigg),
\end{align*}
we observe that only real solutions $(p,x)$ of the system of equations $\partial k_4/\partial x=0$ and $\partial k_4/\partial p=0$ are $(2,1.74724)$, $(2,-1.74724)$, $(-1.36584,-0.835809)$, $(1.36584,-0.835809)$, $(2,-1.74724)$ and $(-0.854598,0.524203)$. Thus, no solution exists in $(0,2)\times (0,1)$, resulting in no critical points. 
%exists of the system of equations $\partial k_4/\partial x=0$ and $\partial k_4/\partial p=0$.

\noindent \underline{On $y=1$} 
\begin{align*}
k_5(p,x):&=\dfrac{1}{21233664}\bigg\{1099 p^6 + 31104 p^2 (4- p^2) + 11484 p^4 (4-p^2) + 
 20736 (4-p^2)^2\\
&\quad\quad \quad\quad\quad\quad+8784 p^2 (4- p^2)^2 - 
 1152 (4-p^2) (1-x^2) \bigg(-16 (8 + x^2) \\ 
&\quad\quad \quad\quad\quad\quad+ p^2 (32 - 27 x + 4 x^2)\bigg) + 96 (4-p^2) (1-x^2) \bigg(48 p x (19 + 4 x)\\
&\quad\quad\quad \quad\quad\quad+ 
   p^3 (149 + 96 x - 48 x^2)\bigg)\bigg\}.
\end{align*}
Simple calculations leads to
\begin{align*}
 \frac{\partial{k_5}}{\partial x} &=\dfrac{(4-p^2)}{110592}\bigg(-192 x (7 + 2 x^2) + 6 p^2 (27 + 56 x - 81 x^2 + 16 x^3)+24 p (19 + 8 x\\
 &\quad\quad\quad\quad\quad\quad  - 57 x^2 - 16 x^3) + 
 p^3 (48 - 197 x - 144 x^2 + 96 x^3)\bigg)
\end{align*}
and \begin{align*}
\dfrac{\partial k_5}{\partial p}&=\dfrac{1}{3538944}\bigg(-1601 p^5 - 3072 x (-19 - 4 x + 19 x^2 + 4 x^3)+768 p (-85 + 54 x   \\
&\quad\quad\quad \quad\quad\quad + 112 x^2- 54 x^3 + 16 x^4)-96 p^3 (-15 + 216 x + 224 x^2- 216 x^3 \\
&\quad\quad\quad \quad\quad\quad  + 32 x^4)-80 p^4 (149 + 96 x - 197 x^2 - 96 x^3 + 48 x^4)+192 p^2 (149 \\
&\quad\quad\quad \quad\quad\quad  - 132 x- 245 x^2 + 132 x^3 + 96 x^4)\bigg).
\end{align*}
We note that the only real solutions $(p,x)$ of the system of equations $\partial k_5/\partial x=0$ and $\partial k_5/\partial p=0$ are $(-5.5858,2.7083)$, $(2,-2.2645)$, $(-2,0.357662)$, $(-1.98983,0.350993)$, $(0.932759,-1.56488)$ and $(-1.03049,-0.27789)$. Thus, no solution exists in $(0,2)\times (0,1)$, resulting in no critical points.					
\item We next examine the maxima attained by $A(p,x,y)$ on the edges of the cuboid $Q$.\\
From (\ref{7b1 9.1}), we have $A(p,0,0)=e_1(p):=(124416 p^2 - 31104 p^4 + 1099 p^6)/21233664.$ It is easy to observe that $e_1'(p)=0$ whenever $p=\delta_0:=0$ and $p=\delta_1:=1.50801\in [0,2]$ as its points of minima and maxima respectively. Hence, 
\begin{equation*}
A(p,0,0)\leq 0.00635802, \quad p\in [0,2].
\end{equation*}

Now considering (\ref{7b1 9.1}) at $y=1,$ we get $A(p,0,1)=e_2(p):=(589824 - 294912 p^2 + 57216 p^3 + 36864 p^4 - 14304 p^5 + 1099 p^6)/21233664.$ It is easy to observe that $e_2'(p)<0$ in $[0,2]$ and hence $p=0$ serves as the point of maxima. So,
\begin{equation*}
A(p,0,1)\leq \dfrac{1}{36}\approx 0.0277778, \quad p\in [0,2].
\end{equation*}
Through computations, (\ref{7b1 9.1}) shows that $A(0,0,y)$ attains its maxima at $y=1.$ This implies that
\begin{equation*}
A(0,0,y)\leq \dfrac{1}{36}, \quad y\in [0,1].
\end{equation*}
Since, (\ref{7b1 9.2}) does not involve $x$, we have $A(p,1,1)=A(p,1,0)=e_3(p):=(331776 + 99072 p^2 - 34704 p^4 - 1601 p^6)/21233664.$ Now, $e_3'(p)=33024p - 23136 p^3 - 1601 p^5=0$ when $p=\delta_2:=0$ and $p=\delta_3:=1.14405$ in the interval $[0,2]$ with $\delta_2$ and $\delta_3$ as points of minima and maxima, respectively. Hence
\begin{equation*}
A(p,1,1)=A(p,1,0)\leq 0.0187629,\quad p\in [0,2].
\end{equation*}
After considering $p=0$ in (\ref{7b1 9.2}), we get, $A(0,1,y)=1/64$. The equation (\ref{5 9.3}) has no variables. So, on the edges, 
the maximum value of $M(p,x,y)$ is
\begin{equation*}
A(2,1,y)=A(2,0,y)=A(2,x,0)=A(2,x,1)=\dfrac{1099}{331776},\quad x,y\in [0,1].
\end{equation*}
Using (\ref{7b1 9.4}), we obtain $A(0,x,1)=e_4(x):=(16 - 14 x^2 + 9 x^3 - 2 x^4)/576.$ Upon calculations, we see that $e_4(x)$ is a decreasing function in $[0,1]$ and attains its maxima at $x=0.$ Thus
\begin{equation*}
A(0,x,1)\leq \dfrac{1}{36},\quad x\in [0,1].
\end{equation*}
Again utilizing (\ref{7b1 9.4}), we get $A(0,x,0)=e_5(x):=x(2-x^2)/64.$ On further calculations, we get $e_5'(x)=0$ for $x=\delta_4:=\sqrt{2/3}.$ Also, $e_5(x)$ is an increases in $[0,\delta_4)$ and decreases in $(\delta_4,1].$ So, it reaches its maximum value at $\delta_4.$ Thus
\begin{equation*}
A(0,x,0)\leq 0.0170103,\quad x\in [0,1].
\end{equation*}
\end{enumerate}
Given all the cases, the inequality (\ref{7b1 s9.5}) holds. Let the function 
$f_0(z)\in \mathcal{S}^{*}_{\mathcal{B}}$, be defined as
\begin{equation}
f_0(z)=z\exp\bigg(\int_{0}^{z}\dfrac{\sqrt{1+\tanh t^3}-1}{t}dt\bigg)=z+\dfrac{z^4}{6}-\dfrac{z^7}{144}+\cdots,\label{7b1 sextremal}
\end{equation}
with $f_0(0)=0$ and $f_0'(0)=1$, acts as an extremal function for the bound of $|H_{3}(1)|$ for $a_2=a_3=a_5=0$ and $a_4=1/6$. 
\end{proof}

Next, we find the sharp bound of $|H_2(3)|$ for functions belonging to the class ${S}^{*}_{\mathcal{B}}$, given by
				
\begin{theorem}				
Let $f\in \mathcal {S}^{*}_{\mathcal{B}}$, then 
\begin{equation}
	|H_{2}(3)|\leq \dfrac{1}{36}.\label{7b2 9.5}
\end{equation}
This bound is sharp.
\end{theorem}
\begin{proof}
We proceed on the similar lines as in the proof of Theorem \ref{7b1 ssharph31}.
Assuming $p_1=:p\in [0,2]$, we substitute the values of $a_ i'(i=3,4,5)$ from (\ref{7b sa2}) and (\ref{7b sa5}) into (\ref{3h23}), we obtain
\begin{align*}
H_{2}(3)&=\dfrac{1}{10616832}\bigg(-761 p^6 + 408 p^4 p_2 - 2880 p^2 p_2^2 - 41472 p_2^3+10848 p^3 p_3 + 52992 p p_2 p_3\\
&\quad\quad \quad \quad \quad \quad-73728 p_3^2 - 31104 p^2 p_4 + 82944 p_2 p_4\bigg).
\end{align*}
Using (\ref{b2})-(\ref{b4}), we arrive at
	$$H_{2}(3)=\dfrac{1}{5308416}\bigg(\beta_5(p,\gamma)+\beta_6(p,\gamma)\eta+\beta_7(p,\gamma){\eta}^2+\beta_8(p,\gamma,\eta)\rho\bigg),$$
	where $\gamma,\eta,\rho\in \mathbb {D},$
\begin{align*}	
\beta_5(p,\gamma):&=-437p^6-5904{\gamma}^2p^2(4-p^2)^2+1440{\gamma}^3p^2(4-p^2)^2+576{\gamma}^4p^2(4-p^2)^2\\
&\quad+5184{\gamma}^2p^2(4-p^2)-852{\gamma}p^4(4-p^2)-4008p^4{\gamma}^2(4-p^2)+1296p^4{\gamma}^3(4-p^2),\\
\beta_6(p,\gamma):&=(1-|\gamma|^2)(4-p^2)(5424p^3-5184p^3{\gamma}-2880p{\gamma}(4-p^2)-2304p\gamma^2 (4-p^2)),\\
\beta_7(p,\gamma):&=576(1-|\gamma|^2)(4-p^2)(-32(4-p^2)-9p^2\bar{\gamma}-4|\gamma|^2(4-p^2))\\
\beta_8(p,\gamma,\eta):&=5184(1-|\gamma|^2)(4-p^2)(1-|\eta|^2)(p^2+4\gamma(4-p^2)).
\end{align*}
Additionally, by using the fact that $|\rho|\leq 1$ and taking $x=|\gamma|$, $y=|\eta|$, we obtain
\begin{align*}
	|H_{2}(3)|\leq \dfrac{1}{10616832}\bigg(|\beta_5(p,\gamma)|+|\beta_6(p,\gamma)|y+|\beta_7(p,\gamma)|y^2+|\beta_8(p,\gamma,\eta)|\bigg)\leq B(p,x,y),
\end{align*}
where 
\begin{equation}
		B(p,x,y)=\dfrac{1}{10616832}\bigg(b_1(p,x)+b_2(p,x)y+b_3(p,x)y^2+b_4(p,x)(1-y^2)\bigg),\label{7b2 snew}
\end{equation} 
with 
\begin{align*}
	b_1(p,x):&=437p^6+5904p^2x^2(4-p^2)^2+1440p^2x^3(4-p^2)^2+576p^2x^4(4-p^2)^2\\
	&\quad +5184p^2x^2(4-p^2)+852p^4x(4-p^2)+4008p^4x^2(4-p^2)+1296p^4x^3(4-p^2),\\
b_2(p,x):&=(4-p^2)(1-x^2)(5424p^3+5184p^3x+2880px(4-p^2)+2304px^2(4-p^2)),\\
b_3(p,x):&=576(4-p^2)(1-x^2)(32(4-p^2)+9p^2x+4x^2(4-p^2)),\\
b_4(p,x):&=5184(4-p^2)(1-x^2)(p^2+4x(4-p^2)).
\end{align*}
At this point, we must maximise $B(p,x,y)$ in the closed cuboid $R:[0,2]\times [0,1]\times [0,1]$. By identifying the maximum values on the twelve edges, the interior of $R$, and the interiors of the six faces, we can prove this.
\begin{enumerate}
\item We start by taking into account, every interior point of $R$. Assume that $(p,x,y)\in (0,2)\times (0,1)\times (0,1).$ On differentiating (\ref{7b2 snew}) with respect to $y$, to locate the points of maxima in the interior of $R$, we obtain
\begin{align*}
	    \dfrac{\partial B}{\partial y}&=\dfrac{(1 - x^2)}{221184}(4 - p^2)  \bigg(48 p x (5 + 4 x) + p^3 (113 + 48 x - 48 x^2)\\
	    & \quad \quad \quad\quad\quad\quad \quad\quad\quad +384 (8 - 9 x + x^2) y - 24 p^2 (41 - 45 x + 4 x^2) y\bigg). 
\end{align*}
Now $\partial B/\partial y=0$ gives
\begin{equation*}
	    y=y_1:=\dfrac{48 p x (5 + 4 x) + p^3 (113 + 48 x - 48 x^2)}{24 (1-x) (16 (x-8) + p^2 (41- 4 x))}.
\end{equation*}
Since $y_1$ must be a member of $(0,1)$ for critical points to exist, this is only possible if
\begin{equation}
48 p x (5 + 4 x) + p^3 (113 + 48 x - 48 x^2)<24 (1-x) (16 (x-8) + p^2 (41- 4 x)).\label{7b2 h1}
\end{equation}
Now, we find the solutions satisfying the inequality (\ref{7b2 h1}) for the existence of critical points using hit and trial method. If $p$ tends to 0 and 2, then no $x\in (0,1)$ exists satisfying (\ref{7b2 h1}). Similarly, if take $x$ tending to 0 and 1, then there does not exist any $p\in (0,2)$ such that (\ref{7b2 h1}) holds.
So, we conclude that the function $B$ has no critical point in $(0,2)\times (0,1)\times (0,1).$ 

\item Now, we study the interior of six faces of the cuboid $R$.\\
\underline{On $p=0$}
\begin{equation}
l_1(x,y):=\dfrac{(1-x^2)(8y^2+x^2y^2+9x(1-y^2))}{288},\label{7b2 9.4}
\end{equation}
with $x,y\in (0,1)$. We note that, in $(0,1)\times(0,1)$, $l_1$ does not have any critical point. As
\begin{equation*}
    \dfrac{\partial l_1}{\partial y}=\dfrac{y(1-x^2)(x-1)(x-8)}{144}\neq 0\quad x,y\in (0,1).
\end{equation*}

\noindent \underline{On $p=2$}
\begin{equation}
    B(2,x,y):=\dfrac{437}{165888}=0.00263431,\quad x,y\in (0,1).\label{7b2 9.3}
\end{equation}
\underline{On $x=0$}
\begin{equation}
    l_2(p,y):=\dfrac{437 p^6 +(4-p^2)( 5424 p^3  y + 18432 (4 - p^2) y^2 + 
 5184 p^2 (1-y^2))}{10616832}\label{7b2 9.1}
\end{equation}
with $p\in (0,2)$ and $y\in (0,1).$ We solve $\partial l_2/\partial p$ and $\partial l_2/\partial y$ to locate the points of maxima. On solving $\partial l_2/\partial y=0,$ we obtain
\begin{equation}
    y=\dfrac{113p^3}{24(41p^2-128)}=:y_p.\label{7b2 y}
\end{equation}
For the given range of $y$, we should have $y_p\in (0,1)$ but no such $p\in (0,2)$ exists.\\
\underline{On $x=1$}
\begin{equation}
   B(p,1,y)= l_3(p,y):=\dfrac{ (147456p^2 - 43920 p^4 + 2201 p^6)}{10616832}, \quad p\in (0,2).\label{7b2 9.2}
\end{equation}
And when computing $\partial l_3/\partial p=0$, $p=p_0\approx1.40378$ turns out to be the critical point. According to elementary calculations, $l_3$ reaches its  maximum value $\approx 0.0128915$ at $p_0.$\\
\noindent \underline{On $y=0$}
\begin{align*}
    l_4(p,x):&=\dfrac{1}{10616832}\bigg(331776 x (1- x^2) + 2304 p^2 (9 - 72 x + 41 x^2 + 82 x^3 + 4 x^4)\\
    &\quad \quad\quad \quad\quad \quad -48 p^4 (108 - 503 x + 650 x^2 + 564 x^3 + 96 x^4)\\
    &\quad \quad\quad \quad \quad\quad+p^6 (437 - 852 x + 1896 x^2 + 144 x^3 + 576 x^4) \bigg)=B(p,x,0).
\end{align*}
On computations, we obtain
\begin{align*}
    \dfrac{\partial l_4}{\partial x}&=\dfrac{(4-p^2)}{884736}\bigg(6912 (1 - 3 x^2) + 96 p^2 (-18 + 41 x + 69 x^2 + 8 x^3)\\
     &\quad \quad\quad \quad \quad \quad-p^4 (-71 + 316 x + 36 x^2 + 192 x^3)\bigg)
\end{align*}
and \begin{align*}
    \dfrac{\partial l_4}{\partial p}&=\dfrac{1}{1769472}\bigg(768 p (9 - 72 x + 41 x^2 + 82 x^3 + 4 x^4)-32 p^3 (108 - 503 x+ 650 x^2\\
    &\quad\quad\quad\quad\quad \quad  + 564 x^3 + 96 x^4)+p^5 (437 - 852 x + 1896 x^2 + 144 x^3 + 576 x^4)\bigg).
\end{align*}
The common real solutions $(p,x)$ of the system of equations, $\partial l_4/\partial x=0$ and $\partial l_4/\partial p=0$ are $(-2,-2.86143)$, $(-2.6516,-0.571214)$, $(-2,-0.247504)$, $(-2,-2.86143)$, $(2,0.0163426)$, $(-2,0.0163426)$, $(2.6516,-0.571214)$, $(2,-0.247504)$, $(0,0.57735)$, $(0,-0.57735)$ and\\ $(-0.914024,0.721238)$. Thus, no solution exists in $(0,2)\times (0,1)$, resulting in no critical points.

\noindent \underline{On $y=1$}
\begin{align*}
l_5(p,x):&=\dfrac{1}{10616832}\bigg(9216 p x (5+4x-5x^2-4x^3) + 36864 (8 - 7 x^2 - x^4)-2304 p^2 (64 \\
&\quad\quad \quad\quad\quad\quad - 9 x- 106 x^2 - x^3 - 12 x^4) - 
 48 p^5 (113 + 48 x - 161 x^2- 48 x^3 + 48 x^4)  \\ 
&\quad\quad \quad\quad\quad\quad+192 p^3 (113 - 12 x - 209 x^2 + 12 x^3 + 96 x^4) + 48 p^4 (384 - 37 x- 1094 x^2 \\
&\quad\quad\quad \quad\quad\quad - 24 x^3 -144 x^4)+p^6 (437 - 852 x + 1896 x^2+ 144 x^3 + 576 x^4)\bigg).
\end{align*}
On computations, we get
\begin{align*}
    \dfrac{\partial l_5}{\partial x}&=\dfrac{(4-p^2)}{884736}\bigg(-1536 x (7 + 2 x^2) + 192 p (5 + 8 x - 15 x^2 - 16 x^3)+48 p^2 (9 + 156 x\\
     &\quad \quad\quad \quad \quad \quad + 3 x^2 + 32 x^3) +8 p^3 (24 - 161 x - 72 x^2 + 96 x^3)+p^4 (71 - 316 x\\
  &\quad \quad\quad \quad \quad \quad  - 36 x^2 - 192 x^3) \bigg)
\end{align*}
and \begin{align*}
    \dfrac{\partial l_5}{\partial p}&=\dfrac{1}{1769472}\bigg(1536 x (5 +4 x - 5 x^2 - 4 x^3) - 
 768 p (64 - 9 x - 106 x^2 - x^3 -12 x^4)\\
&\quad\quad\quad\quad\quad \quad - 40 p^4 (113 + 48 x - 161 x^2 - 48 x^3 + 48 x^4) + 
 96 p^2 (113 - 12 x - 209 x^2 \\
&\quad\quad\quad\quad\quad \quad + 12 x^3 + 96 x^4)+ 32 p^3 (384 - 37 x - 1094 x^2 -24 x^3 - 144 x^4)\\
 &\quad\quad\quad\quad\quad \quad + 
 p^5 (437 - 852 x+ 1896 x^2 + 144 x^3 + 576 x^4)\bigg).
\end{align*}
The common real solutions $(p,x)$ of the system of equations, $\partial l_5/\partial x=0$ and $\partial l_5/\partial p=0$ are $(10.3578,0.237179)$, $(2.61706,3.83978)$, $(-2,-2.97479)$, $(2,2.03127)$, $(-2.52708,-1.62061)$, $(-2,-0.783621)$, $(-1.37805,-1.02453)$, $(-1.77448,0.0488452)$, $(2.00019, -0.515602)$, $(0,0)$ and $(1.35192,-1.00909)$. Thus, no solution exists in $(0,2)\times (0,1)$, resulting in no critical points.

\item Now, we calculate the maximum values achieved by $B(p,x,y)$ on the edges of the cuboid $R$. From (\ref{7b2 9.1}), we have $B(p,0,0)=f_1(p):=(20736 p^2 - 5184 p^4 + 437 p^6)/10616832.$ It is easy to observe that $f_1'(p)=0$ for $p=\delta_0:=0$ in the interval $[0,2]$. The maximum value of $f_1(p)$ is $0.$\\ 
Now considering (\ref{7b2 9.1}) at $y=1,$ we get $B(p,0,1)=f_2(p):=(294912 - 147456 p^2 + 21696 p^3 + 18432 p^4 - 5424 p^5 + 437 p^6)/10616832.$ It is easy to observe that $f_2'(p)$ is a decreasing function in $[0,2]$ and hence $p=0$ acts as its point of maxima. Thus
\begin{equation*}
    B(p,0,1)\leq \dfrac{1}{36}=0.0277778, \quad p\in [0,2].
    \end{equation*}
Through computations, (\ref{7b2 9.1}) shows that $B(0,0,y)=y^2/36$, attains its maximum value at $y=1.$ This implies that
\begin{equation*}
    B(0,0,y)\leq \dfrac{1}{36}, \quad y\in [0,1].
\end{equation*}
Since, (\ref{7b2 9.2}) is independent of $x$, we have $B(p,1,1)=B(p,1,0)=f_3(p):=(147456p^2 - 43920 p^4 + 2201 p^6)/10616832.$ Now, $f_3'(p)=294912 p - 175680 p^3 + 13206 p^5=0$ when $p=\delta_2:=0$ and $p=\delta_3:=1.40378$ in the interval $[0,2]$ with $\delta_2$ and $\delta_3$ as points of minima and maxima respectively. Hence
\begin{equation*}
    B(p,1,1)=B(p,1,0)\leq 0.0128915,\quad p\in [0,2].
\end{equation*}
On substituting $p=0$ in (\ref{7b2 9.2}), we get, $B(0,1,y)=0.$ The equation (\ref{7b2 9.3}) is independent of the all the variables namely $p$, $x$ and $y$. Thus the maximum value of $B(p,x,y)$ on the edges $p=2, x=1; p=2, x=0; p=2, y=0$ and $p=2, y=1,$ respectively, is given by
\begin{equation*}
    B(2,1,y)=B(2,0,y)=B(2,x,0)=B(2,x,1)=\dfrac{437}{165888},\quad x,y\in [0,1].
\end{equation*}
From (\ref{7b2 9.1}), we obtain $B(0,0,y)=y^2/36.$ A simple calculation shows that
\begin{equation*}
    B(0,0,y)\leq \dfrac{1}{36},\quad y\in [0,1].
    \end{equation*}
Using (\ref{7b2 9.4}), we obtain $B(0,x,1)=f_4(x):=(8-7x^2-x^4)/288.$ Upon calculations, we see that $f_4$ is a decreasing function in $[0,1]$ and hence attains its maximum value at $x=0.$ Thus
 \begin{equation*}
     B(0,x,1)\leq \dfrac{1}{36},\quad x\in [0,1].
 \end{equation*}
On again using (\ref{7b2 9.4}), we get $B(0,x,0)=f_5(x):=9x(1-x^2)/288.$ On further calculations, we get $f_5'(x)=0$ for $x=\delta_4:=1/\sqrt{3}.$ Also, $f_5(x)$ is an increasing function in $[0,\delta_4)$ and decreasing in $(\delta_4,1].$ So, it attains its maximum value at $\delta_4$. Thus
 \begin{equation*}
     B(0,x,0)\leq 0.0120281,\quad x\in [0,1].
 \end{equation*}
\end{enumerate}
In view of all the cases, the inequality (\ref{7b2 9.5}) holds.
The function specified in (\ref{7b1 sextremal}) acts as an extremal function for the bounds of $|H_{2}(3)|$ having values $a_3=a_5=0$ and $a_4=1/6.$
\end{proof}
\subsection{Fourth Hankel determinant}

Given that sharp bounds for third-order Hankel determinants have been attained for various subclasses of stalike functions, as shown in Table \ref{7 table1}, determining bounds for fourth-order Hankel determinants proves to be considerably challenging, necessitating extensive computations. Consequently, there have been relatively few efforts in this direction within the existing literature, for recent advancements, we refer \cite{nehapetal, nehacardioid, nehaexpo,bellv}. Subsequently, we introduce a lemma which is required in forthcoming results.

\begin{lemma}\cite{bellv,shelly}\label{2 pomi lemma}
Let $p=1+\sum_{n=1}^{\infty}p_nz^n\in \mathcal{P}.$ Then
\begin{equation*}
|p_n|\leq 2, \quad n\geq 1,\label{2 caratheodory1}
\end{equation*}
\begin{equation*}
|p_{n+k}-\nu p_n p_k|\leq \begin{cases}
2, & 0\leq \nu\leq 1;\\
2|2\nu-1|,& otherwise,
\end{cases}\label{2 caratheodory2}
\end{equation*}
and \begin{equation*}
|p_1^3-\nu p_3|\leq
\begin{cases}2|\nu-4|,& \nu\leq 4/3;\\ \\
2\nu\sqrt{\dfrac{\nu}{\nu-1}},& 4/3<\nu.
\end{cases}\label{2 caratheodory3}
\end{equation*}
\end{lemma}
Now, we try to estimate the bound of sixth and seventh coefficient bounds of function $f\in \mathcal{S}^{*}_{\mathcal{B}}$ as follows:

\begin{lemma}\label{7 a6a7bound}
    Let $f\in \mathcal{S}^{*}_{\mathcal{B}}$, then $|a_6|\leq 0.611233$ and $|a_7|\leq 0.690994$.
\end{lemma}

\begin{proof}
From \eqref{7b sa6}, we have
  \begin{align*} 
1843200a_6&= -2367 p_1^5 - 8560 p_1^3 p_2 + 50880 p_1 p_2^2- 86400 p_1 p_4+ 92160 p_5 -96000 p_2 p_3 \\
&\quad+ 46080 p_1^2 p_3
  \end{align*}
  or 
  \begin{align*} 
1843200|a_6|&\leq | p_1^3(-2367 p_1^2 - 8560 p_2)| + |p_1(50880  p_2^2- 86400 p_4)|+ |92160 p_5 -96000 p_2 p_3|\nonumber \\
&\quad+ |46080 p_1^2 p_3|.%\label{7b a61}
  \end{align*}
  Using Lemma \ref{2 pomi lemma} and the triangle inequality, we arrive at
  \begin{equation}\label{7 a6bound}
      |a_6|\leq \frac{35207}{57600}\approx 0.611233.
  \end{equation}
  Similarly,
  \begin{align*}
530841600a_7&= 601421 p_1^6 - 2365320 p_1^4 p_2 + 4723200 p_2^3- 4818240 p_1^2 p_2^2  - 2648640 p_1^3 p_3- 11980800 p_3^2\nonumber\\
 &\quad +26150400 p_1 p_2 p_3- 
 23500800 p_2 p_4  + 11577600 p_1^2 p_4  - 21012480 p_1 p_5
\end{align*}
or 
\begin{align*}
530841600|a_7|&\leq |p_1^4(601421 p_1^2 - 2365320  p_2)| +|p_2^2 ( 4723200 p_2- 4818240 p_1^2)| \nonumber\\
 &\quad  +|  p_3(2648640 p_1^3- 11980800 p_3)|+|p_2(26150400 p_1 p_3- 23500800  p_4)|\nonumber\\
&\quad  +|p_1( 11577600 p_1^2 p_4  - 21012480  p_5)|.%\label{7b a71}
 \end{align*}
By employing Lemma \ref{2 pomi lemma} and the triangle inequality, we obtain
\begin{equation}\label{7 a7bound}
      |a_7|\leq \frac{31841}{46080}\approx 0.690994.
\end{equation}
\end{proof}

We derive the expression of the fourth Hankel determinant, upon substituting $q=4$ and $n=1$ in (\ref{5hqn}) as follows:
\begin{equation}\label{7 h41}
	H_{4}(1)=a_7H_{3}(1)-a_6B_1+a_5B_2-a_4B_3,
\end{equation}
where
\begin{equation}\label{7t1}
B_1:=a_6(a_3-a_2^2)+a_3(a_2a_5-a_3a_4)-a_4(a_5-a_2a_4),
\end{equation}
\begin{equation}\label{7t2}
B_2:=a_3(a_3a_5-a_4^2)-a_5(a_5-a_2a_4)+a_6(a_4-a_2a_3),
\end{equation}
and 
\begin{equation}\label{7t3}
B_3:=a_4(a_3a_5-a_4^2)-a_5(a_2a_5-a_3a_4)+a_6(a_4-a_2a_3).
\end{equation}
Next, we determine the bound of these $B_i$ for $i=1,2,3$.
By substituting the values from \eqref{7b sa2}-\eqref{7b sa6} in \eqref{7t1}, we get
\begin{align*}
707788800B_1&=110689 p_1^7 - 299496 p_1^5 p_2 - 855360 p_1^4 p3+269760 p_1^3 p_2^2 + 944640 p_1 p_2^3 \\
&\quad - 1128960 p_1^2 p_2 p_3 - 3686400 p_2^2 p_3+4608000 p_1 p_3^2 + 2668800 p_1^3 p_4\\
 &\quad  + 460800 p_1 p_2 p_4 - 3686400 p_3 p_4 - 3870720 p_1^2 p_5+4423680 p_2 p_5
\end{align*}
or
\begin{align*}%\label{7b1value}
707788800|B_1|&\leq |p_1^4(110689 p_1^7- 855360  p_3)|+|p_2^2(269760 p_1^3  - 3686400  p_3)|+|4608000 p_1 p_3^2|\nonumber\\
&\quad +|p_1 p_2( 944640 p_2^2  + 460800  p_4)| +|p_1^2 p_2(- 299496 p_1^3- 1128960  p_3)|\nonumber\\
 &\quad  +| p_4( 2668800 p_1^3- 3686400 p_3)| +|p_5(4423680 p_2- 3870720 p_1^2)|
\end{align*}
Using Lemma \ref{2 pomi lemma} and the triangle inequality, we get
\begin{align}\label{7b1valuefinal}
|B_1|&\leq \frac{1}{707788800}\bigg(110590464 + 255027456\sqrt{\frac{165}{744671}}+471859200\sqrt{\frac{15}{3559}}+117964800 \sqrt{\frac{3}{53}} \bigg)\nonumber\\
&\approx 0.244544.
\end{align}
Similarly, 
\begin{align*}
  5662310400B_2&=  927648 p_1^5 p_3-149225 p_1^8 + 785840 p_1^6 p_2+3939840 p_1^3 p_2 p_3  - 2073600 p_1^2 p_2^3 - 8294400 p_2^4\\
  &\quad + 23316480 p_1 p_2^2 p_3- 2025600 p_1^4 p_2^2  - 11704320 p_1^2 p_3^2 - 
 29491200 p_2 p_3^2 + 25804800 p_1 p_3 p_4\\
 &\quad - 3676800 p_1^4 p_4-3225600 p_1^2 p_2 p_4- 29491200 p_1 p_2 p_5  + 27648000 p_2^2 p_4  - 
 22118400 p_4^2  \\
 &\quad+ 6144000 p_1^3 p_5+ 23592960 p_3 p_5
\end{align*}
or
\begin{align*}
  5662310400|B_2|&\leq  |p_1^5(927648  p_3-149225 p_1^3)| + |p_1^3 p_2(785840 p_1^3+3939840  p_3)|\nonumber\\
  &\quad +| p_2^3( - 2073600 p_1^2 - 8294400 p_2)| + | p_1 p_2^2(23316480 p_3- 2025600 p_1^3)|\nonumber \\
 &\quad +|p_3^2( - 11704320 p_1^2  - 29491200 p_2)| + |p_1p_4(25804800 p_3- 3676800 p_1^3)|\nonumber  \\
 &\quad+|p_1p_2(-3225600 p_1 p_4- 29491200 p_5)|  + |p_4(27648000 p_2^2  - 22118400 p_4)|\nonumber\\
 &\quad+|p_5( 6144000 p_1^3+ 23592960 p_3 |.%\label{7b2value1}
\end{align*}
With the aid of Lemma \ref{2 pomi lemma} and the triangle inequality, we obtain
\begin{align}\label{7b2valuefinal}
    |B_2|&\leq \frac{1}{5662310400}\bigg(1982371840+\frac{22263552}{149225}\sqrt{\frac{6442}{778423}}+\frac{43008}{383}\sqrt{\frac{42}{2305}}+ \frac{97152}{1055}\sqrt{\frac{759}{11089}}\bigg)\nonumber\\
    &\approx 0.350099.
\end{align}
and 
\begin{align*}
305764761600B_3&= 133166592 p_1^5 p_3 -8521200 p_1^8 +126489600 p_1^4 p_2 p_3 - 621000 p_1^7 p_2+331084800 p_1^4 p_2^2\\
&\quad + 2362245120 p_1 p_2^2 p_3+6566400 p_1^3 p_2^3 - 199065600 p_2^3 p_3 + 99532800 p_1 p_2^4\\
&\quad- 49766400 p_1 p_2^2 p_4   - 
 1260195840 p_1^3 p_2 p_3- 1327104000 p_2 p_3^2- 103680000 p_1^2 p_2^2 p_3\\
 &\quad   - 20260800 p_1^5 p_2^2- 177638400 p_1^3 p_3^2-176947200 p_3^3 + 398131200 p_1^2 p_3 p_4 \\
 &\quad- 12182400 p_1^5 p_4+ 331776000 p_1^3 p_5+ 1274019840 p_3 p_5+398131200 p_2 p_3 p_4\\
 &\quad - 311040000 p_1^4 p_4 - 1194393600 p_1 p_3 p_4 +  1492992000 p_1^2 p_2 p_4- 1592524800 p_1 p_2 p_5\\
 &\quad +637009920 p_1^2 p_3^2- 160625 p_1^9- 879206400 p_1^2 p_2^3+ 99532800 p_1 p_2 p_3^2- 298598400 p_1 p_4^2
\end{align*}
or
\begin{align*}%\label{7b3value1}
305764761600|B_3|&\leq |p_1^5( 133166592  p_3 -8521200 p_1^3)| +|p_1^4 p_2(126489600  p_3- 621000 p_1^3)|\nonumber\\
&\quad+|p_1 p_2^2(331084800 p_1^3 + 2362245120  p_3)|+|p_2^3(6566400 p_1^3 - 199065600  p_3)|\nonumber \\
&\quad+ |p_1 p_2^2(99532800 p_2^2- 49766400  p_4 )| +|p_2p_3( - 
 1260195840 p_1^3- 1327104000 p_3)|\nonumber\\
 &\quad +|p_1^2 p_2^2( - 103680000  p_3 - 20260800 p_1^3)|+|p_3^2(- 177638400 p_1^3 -176947200 p_3)|\nonumber\\
 &\quad +|p_1^2p_4( 398131200 p_3- 12182400 p_1^3)|+|p_5( 331776000 p_1^3+ 1274019840 p_3)|\nonumber\\
  &\quad  +|p_2p_4(398131200 p_3-95385600 p_1^3)|+|p_1p_4(- 311040000 p_1^3 - 1194393600 p_3)|\nonumber\\
  &\quad +|p_1 p_2 ( 1492992000 p_1 p_4- 1592524800 p_5)|+|637009920 p_1^2 p_3^2|+|160625 p_1^9|\nonumber\\
  &\quad+|879206400 p_1^2 p_2^3|+ |99532800 p_1 p_2 p_3^2|+| 298598400 p_1 p_4^2|.
\end{align*}
Through Lemma \ref{2 pomi lemma} and the triangle inequality, we have
\begin{align}\label{7b3valuefinal}
        |B_3|&\leq \frac{1}{305764761600}\bigg(210371822080 + 147726139392\sqrt{\frac{114}{32059}}+101921587200\sqrt{\frac{6}{1489}}\nonumber\\
        &\quad\quad\quad\quad\quad\quad\quad\quad+64762675200\sqrt{\frac{366}{23309}}+12740198400 \sqrt{\frac{6}{73}}+  \frac{76441190400}{\sqrt{557}}  \bigg)\nonumber\\
        &\approx 0.787068.
    \end{align}
Upon substituting values from \eqref{7 a6bound}, \eqref{7 a7bound}, \eqref{7b1 ssharph31}, \eqref{7b1valuefinal}-\eqref{7b3valuefinal}, and Theorem \ref{7 initialbound} in \eqref{7 h41}, we obtain the following result, given by
\begin{theorem}
Let $f\in \mathcal{S}_{\mathcal{B}}^{*}$, then $|H_4(1)|\leq 0.169251$.
\end{theorem}
				
%%%%%%%%%%%%%%%%%%convex case%%%%%%%%%%%%%%%%%%%
\section{Sharp Hankel Determinants for $\mathcal{C}_{\mathcal{B}}$}
In this section, we determine the sharp bounds for the second and third-order Hankel determinants for functions $f\in \mathcal{C}_{\mathcal{B}}$. Below, we provide the expressions for the initial coefficients of functions 
$f\in \mathcal{C}_{\mathcal{B}}$ in terms of Carath\'{eodory} coefficients, serving as a foundation for subsequent calculations.

Let $f\in \mathcal{C}_{\mathcal{B}},$ then there exists a Schwarz function $w(z)$ such that
\begin{equation}\label{7 cformulaai}
	1+\dfrac{zf''(z)}{f'(z)}=\sqrt{1+\tanh w(z)}.
\end{equation}
Suppose that $p(z)=1+p_1z+p_2z^2+\cdots\in \mathcal{P}$ and consider $w(z)=(p(z)-1)/(p(z)+1)$. Further, by substituting the expansions of $w(z)$, $p(z)$ and $f(z)$ in (\ref{7 cformulaai}) and then comparing the coefficients, we obtain the expressions of $a_i (i=2,3,...,7)$ in terms of $p_j (j=1,2,...,5)$, given by
\begin{equation}\label{7b ca2}
a_2=\dfrac{1}{8}p_1,\quad a_3=\dfrac{1}{192}\bigg(8p_2-3p_1^2\bigg), \quad a_4=\dfrac{1}{9216}\bigg(23p_1^3 -168p_1 p_2 + 192 p_3\bigg)
\end{equation}
and
\begin{equation}\label{7b ca5}
a_5=\dfrac{1}{92160}\bigg( -11p_1^4+528 p_1^2 p_2-576p_2^2 -1056 p_1 p_3+1152 p_4\bigg).
\end{equation}
 
The following theorem presents the sharp bound for $|H_{3}(1)|$ for functions belonging to the class $\mathcal{C}_{\mathcal{B}}$.

\begin{theorem}\label{7b2 csharph31}
Let $f\in \mathcal {C}_{\mathcal{B}}$, then 
\begin{equation}
|H_{3}(1)|\leq 1/576.\label{7b1 c9.5}
\end{equation}
This result is sharp.
\end{theorem}

\begin{proof}  
Since the class $\mathcal {P}$ is invariant under rotation, the value of $p_1$ belongs to the interval [0,2]. Let $p:=p_1$ and then substitute the values of $a_i(i=2,3,4,5)$ in (\ref{1h3}) from (\ref{7b ca2}) and (\ref{7b ca5}). We get
\begin{align*}
H_{3}(1)&=\dfrac{1}{424673280}\bigg(-3581 p^6 - 11184 p^4 p_2 - 2880 p^2 p_2^2 - 141312 p_2^3+73344 p^3 p_3 + 211968 p p_2 p_3\\
&\quad \quad\quad \quad \quad \quad\quad  - 184320 p_3^2 - 165888 p^2 p_4 + 221184 p_2 p_4\bigg).
\end{align*}
After simplifying the calculations through (\ref{b2})-(\ref{b4}), we obtain				
$$H_{3}(1)=\dfrac{1}{424673280}\bigg(\beta_9(p,\gamma)+\beta_{10}(p,\gamma)\eta+\beta_{11}(p,\gamma){\eta}^2+\beta_{12}(p,\gamma,\eta)\rho\bigg),$$
for $\gamma,\eta,\rho\in \mathbb {D}.$ Here
\begin{align*}
\beta_9(p,\gamma):&=-1157p^6-5328{\gamma}^2p^2(4-p^2)^2-15360\gamma^3(4-p^2)^2-4224{\gamma}^3p^2(4-p^2)^2+2304{\gamma}^4p^2(4-p^2)^2\\
&\quad+3144{\gamma}p^4(4-p^2)-1056p^4{\gamma}^2(4-p^2)-6912p^4{\gamma}^3(4-p^2)-27648{\gamma}^2p^2(4-p^2),\\
\beta_{10}(p,\gamma):&=192(1-|\gamma|^2)(4-p^2)(83p^3+144{\gamma}p^3+84p\gamma(4-p^2)-48p{\gamma}^2(4-p^2)),\\
\beta_{11}(p,\gamma):&=9216(1-|\gamma|^2)(4-p^2)(-5(4-p^2)-|\gamma|^2(4-p^2)+3p^2\bar{\gamma}),\\
\beta_{12}(p,\gamma,\eta):&=27648(1-|\gamma|^2)(4-p^2)(1-|\eta|^2)(2(4-p^2)\gamma-p^2).
\end{align*}
By choosing $x=|\gamma|$, $y=|\eta|$ and utilizing the fact that $|\rho|\leq 1,$ the above expression reduces to the following:
\begin{align*}
|H_{3}(1)|\leq \dfrac{1}{424673280}\bigg(|\beta_9(p,\gamma)|+|\beta_{10}(p,\gamma)|y+|\beta_{11}(p,\gamma)|y^2+|\beta_{12}(p,\gamma,\eta)|\bigg)\leq C(p,x,y),
\end{align*}
where 
\begin{equation}
C(p,x,y)=\dfrac{1}{424673280}\bigg(c_1(p,x)+c_2(p,x)y+c_3(p,x)y^2+c_4(p,x)(1-y^2)\bigg),\label{7b1 cnew}
\end{equation} 
with 
\begin{align*}
c_1(p,x):&=1157p^6+5328x^2p^2(4-p^2)^2+15360x^3(4-p^2)^2+4224x^3p^2(4-p^2)^2+2304x^4p^2(4-p^2)^2\\
&\quad +3144xp^4(4-p^2)+1056p^4x^2(4-p^2)+6912p^4x^3(4-p^2)+27648x^2p^2(4-p^2),\\
c_2(p,x):&=192(1-x^2)(4-p^2)(83p^3+144xp^3+84px(4-p^2)+48px^2(4-p^2)),\\
c_3(p,x):&=9216(1-x^2)(4-p^2)(5(4-p^2)+x^2(4-p^2)+3p^2x),\\
c_4(p,x):&=27648(1-x^2)(4-p^2)(2x(4-p^2)+p^2).
\end{align*}
	
In the closed cuboid $S:[0,2]\times [0,1]\times [0,1]$, we now maximise $C(p,x,y)$, by locating the maximum values in the interior of the six faces, on the twelve edges, and in the interior of $S$.
\begin{enumerate}
\item We start by taking into account every internal point of $S$. Assume that $(p,x,y)\in (0,2)\times (0,1)\times (0,1)$. We calculate $\partial{C}/\partial y$, partially differentiate (\ref{7b1 cnew}) with respect to $y$ to identify the points of maxima in the interior of $U$. We get
\begin{align*}
\dfrac{\partial C}{\partial y}&=\dfrac{(4 - p^2)(1 - x^2)}{2211840}  \bigg(48 p x (7 + 4 x) + p^3 (83+60 x -48 x^2) -
 96 p^2 (8 - 9 x + x^2) y \\
&\quad \quad\quad \quad\quad\quad\quad\quad\quad+384 (5 - 6 x + x^2) y\bigg). 
\end{align*}
Now $\partial C/\partial y=0$ gives
\begin{equation*}
y=y_0:=\dfrac{48 p x (7 + 4 x) + p^3 (83 + 60 x - 48 x^2)}{96(1-x) (p^2 (8-x) - 4 (5-x))}.
\end{equation*}
The existence of critical points require that $y_0$ belong to $(0,1)$, which is only possible when 					
\begin{equation}
83 p^3 + 336 p x + 60 p^3 x + 192 p x^2 - 48 p^3 x^2<96 (-1 + x) (20 - 8 p^2 - 4 x + p^2 x).\label{7b1 ch1}
\end{equation}		
Now, we find the solution satisfying the inequality (\ref{7b1 ch1})
for the existence of critical points using the hit and trial method. If we assume $p$ tends to $0$, then there does not exist any $x\in (0,1)$ satisfying (\ref{7b1 ch1}). But, when $p$ tends to 2, (\ref{7b1 ch1}) holds only when $x<61/288.$ 
Similarly, if we assume $x$ tends to 0, then for all $p>1.75665$, (\ref{7b1 ch1}) holds. After calculations, we observe that there does not exist any $x\in (0,1)$ when $p\in (0,1.75665).$ Thus, the domain for the solution is $(1.75665,2)\times (0,61/288).$ Now, we examine that $\frac{\partial C}{\partial y}|_{y=y_0}\neq 0$ in $(1.75665,2)\times (0,61/288).$ So, we conclude that the function $M$ has no critical point in $(0,2)\times (0,1)\times (0,1).$ 
											
\item The interior of each of the cuboid $S$'s six faces is now being considered.\\
\underline{On $p=0$}
\begin{equation}
m_1(x,y):=\dfrac{20 x^3+ 12 (1-x^2) (5 + x^2) y^2 + 72 x (1-x^2) (1- y^2)}{34560},\quad x,y\in (0,1).\label{7b1 c9.4}
\end{equation}
 Since
\begin{equation*}
\dfrac{\partial m_1}{\partial y}=\dfrac{(1 - x)^2(x+1)(5-x)y}{1440}\neq 0,\quad x,y\in (0,1),
\end{equation*}
indicates that $m_1$ has no critical points in $(0,1)\times(0,1)$.	\\					
\noindent \underline{On $p=2$}
\begin{equation}
C(2,x,y):=\dfrac{1157}{6635520}\approx 0.000174365,\quad x,y\in (0,1).\label{7b1 c9.3}
\end{equation}
\noindent \underline{On $x=0$}%, C(p,x,y)$ becomes
\begin{align}
m_2(p,y):&=\dfrac{1157 p^6 +(4-p^2)(15936 p^3  y + 46080 (4-p^2) y^2 + 
 27648 p^2 (1 -y^2))}{424673280} \label{7b1 c9.1}
\end{align}
with $p\in (0,2)$ and $y\in (0,1).$ To determine the points of maxima, we solve $\partial m_2/\partial p=0$ and $\partial m_2/\partial y=0$. After solving $\partial m_2/\partial y=0,$ we get
\begin{equation}
y=\dfrac{83p^3}{384(2p^2-5)}=:y_p.\label{7b1 cy}
\end{equation}
In order to have $y_p\in (0,1)$ for the given range of $y$, $p_0:=p>\approx 1.75665$ is required. Based on calculations, $\partial m_2/\partial p=0$ gives
\begin{equation}
36864 p - 18432 p^3 + 1157 p^5 + 31872 p^2 y - 13280 p^4 y - 
 159744 p y^2 + 49152 p^3 y^2=0.\label{7b1 c9}
\end{equation}
On substituting equation (\ref{7b1 cy}) into equation (\ref{7b1 c9}), we have
\begin{equation}
1843200p - 2396160 p^3 + 1021152 p^5 - 152402 p^7 + 2367 p^9=0.\label{7b1 c40}
\end{equation}
A numerical calculation suggests that $p\approx 1.28894\in (0,2)$ is the solution of (\ref{7b1 c40}). So, we conclude that $m_2$ does not have any critical point in $(0,2)\times(0,1)$.\\					
\noindent \underline{On $x=1$}
\begin{equation}
C(p,1,y)=m_3(p,y):=\dfrac{245760 + 177408 p^2 - 62688 p^4 + 1901 p^6}{424673280}, \quad p\in (0,2).\label{7b1 c9.2}
\end{equation}
While computing $\partial m_3/\partial p=0$, $p_0:=p\approx 1.23293$ comes out to be the critical point. After simple calculations, $m_3$ achieves its maximum value $\approx 0.000888358$ at $p_0.$\\	
\noindent \underline{On $y=0$} 
\begin{align*}
m_4(p,x):&=\dfrac{1}{424673280}\bigg(49152 x (18- 13 x^2)+
 2304 p^2 (48 - 192 x + 37 x^2 + 168 x^3\\
&\quad \quad\quad\quad\quad\quad\quad  + 16 x^4)-96 p^4 (288 - 707 x + 400 x^2 + 480 x^3 + 192 x^4) \\
&\quad\quad \quad\quad\quad\quad\quad +p^6 (1157 - 3144 x + 4272 x^2 - 2688 x^3 + 2304 x^4)\bigg).
\end{align*}
After further calculations such as
\begin{align*}
\dfrac{\partial m_4}{\partial x}&=\dfrac{(4-p^2)}{17694720}\bigg(1536 (6 - 13 x^2)-48 p^2 (48 - 37 x - 148 x^2 - 32 x^3)\\
&\quad\quad \quad\quad\quad\quad  \quad+p^4 (131 - 356 x + 336 x^2 - 384 x^3)\bigg)
\end{align*}
and 
\begin{align*}
\dfrac{\partial m_4}{\partial p}&=\dfrac{1}{70778880}\bigg(768 p (48 - 192 x + 37 x^2 + 168 x^3 + 16 x^4)-64 p^3 (288 - 707 x \\
&\quad\quad\quad \quad\quad\quad  + 400 x^2 + 480 x^3 + 192 x^4)+p^5 (1157 - 3144 x + 4272 x^2- 2688 x^3 \\
&\quad\quad\quad \quad\quad\quad + 2304 x^4)\bigg),
\end{align*}
we observe that only real solutions $(p,x)$ of the system of equations $\partial m_4/\partial x=0$ and $\partial m_4/\partial p=0$ are $(3.85748,0.257377)$, $(3.76933,0.0851082)$, $(-3.76933,0.0851082)$, $(2,-0.644779)$, $(-3.85748,0.257377)$, $(-1.57205,-1.03976)$, $(-2,-0.644779)$, $(0,-0.679366)$, $(0,0.679366)$, $(-1.12904,0.941715)$ and $(1.57205,-1.03976)$. Thus, no solution exists in $(0,2)\times (0,1)$, resulting in no critical points. 

\noindent \underline{On $y=1$} 
\begin{align*}
m_5(p,x):&=\dfrac{1}{424673280}\bigg\{1157 p^6 + 3144 p^4 (4 - p^2) x + 27648 p^2 (4-p^2) x^2\\
&\quad\quad \quad\quad\quad\quad +1056 p^4 (4 -p^2) x^2+ 5328 p^2 (4 - p^2)^2 x^2 + 
 6912 p^4 (4-p^2) x^3  \\ 
&\quad\quad \quad\quad\quad\quad+ 15360 (4-p^2)^2 x^3 + 4224 p^2 (4-p^2)^2 x^3+ 
 2304 p^2 (4-p^2)^2 x^4\\
&\quad\quad\quad \quad\quad\quad - 
   9216 (4 - p^2) (1- x^2) (-4 (5 + x^2)+ p^2 (5 - 3 x + x^2)) \\
&\quad\quad\quad \quad\quad\quad+ 192 (4 - p^2) (1-x^2) (48 p x (7 + 4 x)+ 
   p^3 (83 + 60 x - 48 x^2)) \bigg\}.
\end{align*}
After calculations, 
\begin{align*}
  \dfrac{\partial m_5}{\partial p}&= \frac{1}{70778880} \bigg(6144 x (7 +4 x - 7 x^2 - 4 x^3) - 
 768 p (160 -48 x - 213 x^2 + 72 x^3 - 48 x^4)\\
  &\quad\quad\quad \quad\quad\quad-160 p^4 (83 + 60 x - 131 x^2 - 60 x^3 + 48 x^4) +
 384 p^2 (83 - 24 x - 179 x^2 \\
 &\quad\quad\quad \quad\quad\quad + 24 x^3 + 96 x^4)+64 p^3 (480 - 157 x - 1072 x^2 + 384 x^3 - 288 x^4) \\
&\quad\quad\quad \quad\quad\quad + 
 p^5 (1157 - 3144 x + 4272 x^2- 2688 x^3 + 2304 x^4)
  \bigg)
\end{align*}
and 
\begin{align*}
  \dfrac{\partial m_5}{\partial x} &=\frac{(4-p^2)}{17694720} \bigg(1536 x ( 5 x-8 - 4 x^2) + 384 p (7 + 8 x - 21 x^2 - 16 x^3) +48 p^2 (24 \\
  &\quad\quad\quad \quad\quad\quad
  + 149 x- 68 x^2 + 64 x^3) + 
 16 p^3 (30 - 131 x - 90 x^2 + 96 x^3)\\
 &\quad\quad\quad \quad\quad\quad +p^4 (131 - 356 x + 336 x^2 - 384 x^3)
  \bigg),
\end{align*}
we observe that only real solutions $(p,x)$ of the system of equations $\partial m_5/\partial x=0$ and $\partial m_5/\partial p=0$ are $(27.1136,0.413453)$, $(-1.50007,-6.38485)$, $(2.68178,5.28944)$, $(2,1.39993)$, $(-1.80124,-0.775344)$, $(-4.98296,1.51367)$, $(-2.89665,3.0885)$, $(2,-0.192442)$, $(0,0)$,\\ $(-1.74974,0.159528)$, $(-2,0.737618)$, $(-0.776187,0.905316)$, $(1.6981,-0.807967)$, $(2,-1.20749)$ and $(1.5135,-0.953552)$. Thus, no solution exists in $(0,2)\times (0,1)$, resulting in no critical points. 
			
\item We next examine the maxima attained by $C(p,x,y)$ on the edges of the cuboid $S$.\\
From (\ref{7b1 c9.1}), we have $M(p,0,0)=g_1(p):=(110592 p^2 - 27648 p^4 + 1157 p^6)/424673280.$ 
It is easy to observe that $g_1'(p)=0$ whenever $p=\delta_0:=0$ and $p=\delta_1:=1.53142\in [0,2]$ as its points of minima and maxima respectively.  
Hence, 
\begin{equation*}
C(p,0,0)\leq 0.0002878, \quad p\in [0,2].
\end{equation*}

\noindent Now consider (\ref{7b1 c9.1}) at $y=1,$ we get $C(p,0,1)=g_2(p):=(737280 - 368640 p^2 + 63744 p^3 + 46080 p^4 - 15936 p^5 + 1157 p^6)/424673280.$ It is easy to observe that $g_2'(p)<0$ in $[0,2]$ and hence $p=0$ serves as the point of maxima. So,
\begin{equation*}
C(p,0,1)\leq \dfrac{1}{576}\approx 0.00173611, \quad p\in [0,2].
\end{equation*}
Through computations, (\ref{7b1 c9.1}) shows that $C(0,0,y)=y^2/576$ attains its maxima at $y=1.$ This implies that
\begin{equation*}
C(0,0,y)\leq \dfrac{1}{36}, \quad y\in [0,1].
\end{equation*}
Since, (\ref{7b1 c9.2}) does not involve $x$, we have $C(p,1,1)=C(p,1,0)=g_3(p):=(245760 + 177408 p^2 - 62688 p^4 + 1901 p^6)/424673280.$ Now, $g_3'(p)=354816 p - 250752 p^3 + 11406 p^5=0$ when $p=\delta_2:=0$ and $p=\delta_3:=1.23293$ in the interval $[0,2]$ with $\delta_2$ and $\delta_3$ as points of minima and maxima, respectively. Hence
\begin{equation*}
C(p,1,1)=C(p,1,0)\leq 0.000888358,\quad p\in [0,2].
\end{equation*}
After considering $p=0$ in (\ref{7b1 c9.2}), we get, $C(0,1,y)=1/1728\approx 0.000578704.$ \\
The equation (\ref{7b1 c9.3}) has no variables. So, on the edges, 
the maximum value of $C(p,x,y)$ is
\begin{equation*}
C(2,1,y)=C(2,0,y)=C(2,x,0)=C(2,x,1)=\dfrac{1157}{6635520},\quad x,y\in [0,1].
\end{equation*}

\noindent Using (\ref{7b1 c9.4}), we obtain $C(0,x,1)=g_4(x):=(15 - 12 x^2 + 5 x^3 - 3 x^4)/8640.$ Upon calculations, we see that $g_4(x)$ is a decreasing function in $[0,1]$ and attains its maxima at $x=0.$ Hence
\begin{equation*}
C(0,x,1)\leq \dfrac{1}{576},\quad x\in [0,1].
\end{equation*}
Again utilizing (\ref{7b1 c9.4}), we get $C(0,x,0)=g_5(x):=x(18-13x^2)/8640.$ On further calculations, we get $g_5'(x)=0$ for $x=\delta_4:=\sqrt{6/13}.$ Also, $g_5(x)$ increases in $[0,\delta_4)$ and decreases in $(\delta_4,1].$ So, it reaches its maximum value at $\delta_4.$ Thus
\begin{equation*}
C(0,x,0)\leq0.000943564,\quad x\in [0,1].
\end{equation*}
\end{enumerate}
Given all the cases, the inequality (\ref{7b1 c9.5}) holds.
Let the function
$f_1(z)\in \mathcal{C}_{\mathcal{B}}$, be defined as
\begin{equation}
f_1(z)=\int_{0}^{z}\bigg(\exp\bigg(\int_{0}^{t}\dfrac{\sqrt{1+\tanh u^3}-1}{u}du\bigg)\bigg)dt=z+\dfrac{z^4}{24}-\dfrac{z^7}{1008}-\cdots,\label{7b1 cextremal}
\end{equation}
with $f_1(0)=0$ and $f_1'(0)=1$, acts as an extremal function for the bound of $|H_{3}(1)|$ for $a_2=a_3=a_5=0$ and $a_4=1/24$. 
\end{proof}

\begin{theorem}				
Let $f\in \mathcal {C}_{\mathcal{B}}$, then 
\begin{equation}
	|H_{2}(3)|\leq \dfrac{1}{576}.\label{7b2 c9.5}
\end{equation}
This bound is sharp.
\end{theorem}
\begin{proof}
We proceed on the similar lines as in the proof of Theorem \ref{7b1 ssharph31}.
Assuming $p_1=:p\in [0,2]$, we substitute the values of $a_ i's(i=3,4,5)$ from (\ref{7b ca2}) and (\ref{7b ca5}) in (\ref{3h23}), we obtain
\begin{align*}
H_{2}(3)&=\dfrac{1}{424673280}\bigg(-1853 p^6 - 1488 p^4 p_2 + 1728 p^2 p_2^2 - 110592 p_2^3+31872 p^3 p_3\\
&\quad\quad \quad \quad \quad \quad \quad+ 119808 p p_2 p_3 - 184320 p_3^2 - 82944 p^2 p_4 + 221184 p_2 p_4\bigg).
\end{align*}
Using (\ref{b2})-(\ref{b4}) for simplification, we arrive at
	$$H_{2}(3)=\dfrac{1}{424673280}\bigg(\beta_{13}(p,\gamma)+\beta_{14}(p,\gamma)\eta+\beta_{15}(p,\gamma){\eta}^2+\beta_{16}(p,\gamma,\eta)\rho\bigg),$$
	where $\gamma,\eta,\rho\in \mathbb {D},$
\begin{align*}	
\beta_{13}(p,\gamma):&=-1109p^6-15696{\gamma}^2p^2(4-p^2)^2+3456{\gamma}^3p^2(4-p^2)^2+2304{\gamma}^4p^2(4-p^2)^2\\
&\quad+13824{\gamma}^2p^2(4-p^2)-2376{\gamma}p^4(4-p^2)-10272p^4{\gamma}^2(4-p^2)+3456p^4{\gamma}^3(4-p^2),\\
\beta_{14}(p,\gamma):&=192(1-|\gamma|^2)(4-p^2)(71p^3-72p^3{\gamma}-36p{\gamma}(4-p^2)-48p\gamma^2 (4-p^2)),\\
\beta_{15}(p,\gamma):&=4608(1-|\gamma|^2)(4-p^2)(-10(4-p^2)-3p^2\bar{\gamma}-2|\gamma|^2(4-p^2))\\
\beta_{16}(p,\gamma,\eta):&=13824(1-|\gamma|^2)(4-p^2)(1-|\eta|^2)(p^2+4\gamma(4-p^2)).
\end{align*}
Additionally, by using the fact that $|\rho|\leq 1,$ and taking $x=|\gamma|$, $y=|\eta|$, we have
\begin{align*}
	|H_{2}(3)|\leq \dfrac{1}{424673280}\bigg(|\beta_{13}(p,\gamma)|+|\beta_{14}(p,\gamma)|y+|\beta_{15}(p,\gamma)|y^2+|\beta_{16}(p,\gamma,\eta)|\bigg)\leq D(p,x,y),
\end{align*}
where 
\begin{equation}
		D(p,x,y)=\dfrac{1}{424673280}\bigg(d_1(p,x)+d_2(p,x)y+d_3(p,x)y^2+d_4(p,x)(1-y^2)\bigg),\label{7b2 cnew}
\end{equation} 
with 
\begin{align*}
	d_1(p,x):&=1109p^6+15696p^2x^2(4-p^2)^2+3456p^2x^3(4-p^2)^2+2304p^2x^4(4-p^2)^2\\
	&\quad +13824p^2x^2(4-p^2)+2376p^4x(4-p^2)+10272p^4x^2(4-p^2)+3456p^4x^3(4-p^2),\\
d_2(p,x):&=192(4-p^2)(1-x^2)(71p^3+72p^3x+36px(4-p^2)+48px^2(4-p^2)),\\
d_3(p,x):&=4608(4-p^2)(1-x^2)(10(4-p^2)+3p^2x+2x^2(4-p^2)),\\
d_4(p,x):&=13824(4-p^2)(1-x^2)(p^2+4x(4-p^2)).
\end{align*}
Now, we must maximise $D(p,x,y)$ in the closed cuboid $T:[0,2]\times [0,1]\times [0,1]$. By identifying the maximum values on the twelve edges, the interior of $T$, and the interiors of the six faces, we can prove this.
\begin{enumerate}
\item We start by taking into account, every interior point of $T$. Assume that $(p,x,y)\in (0,2)\times (0,1)\times (0,1).$ We partially differentiate (\ref{7b2 cnew}) with respect to $y$ in order to locate the points of maxima in the interior of $T$. We obtain
\begin{align*}
	    \dfrac{\partial D}{\partial y}&=\dfrac{(1 - x^2)}{2211840}(4 - p^2)  \bigg(48 p x (3 + 4 x) + p^3 (71+36 x-48 x^2)\\
	    & \quad \quad \quad\quad\quad\quad \quad\quad\quad +384 (5 - 6 x + x^2) y-48 p^2 (13 - 15 x + 2 x^2) y\bigg). 
\end{align*}
Now $\partial D/\partial y=0$ gives
\begin{equation*}
	    y=y_1:=\dfrac{48 p x (3 + 4 x) + p^3 (71 + 36 x - 48 x^2)}{48 (1-x) (8 (x-5) + p^2 (13 - 2 x))}.
\end{equation*}
Since $y_1$ must be a member of $(0,1)$ for critical points to exist, this is only possible if
\begin{equation}
71 p^3 +px( 144  + 36 p^2+ 192 x - 48 p^2 x)<48 (-40 + 13 p^2 +x( 48  - 15 p^2 - 8 x+ 2 p^2 x)).\label{7b2 ch1}
\end{equation}
Now, we find the solutions satisfying the inequality (\ref{7b2 ch1}) 
for the existence of critical points using hit and trial method. If $p$ tends to $2$, then (\ref{7b2 ch1}) holds whenever $x<1/144$. Also, no such $x\in(0,1)$, satisfying (\ref{7b2 ch1}) when $p$ tends to $0$. Similarly, if take $x$ tending to $0$, then (\ref{7b2 ch1})  holds for $p>1.99513$ only, whereas 
there does not exist any $p\in (0,2)$ such that equation (\ref{7b2 ch1}) holds when $x$ tends to $1$.
Thus, the domain for the solution of the equation is $(1.99513,2)\times (0,1/144).$ Now, we examine that $\frac{\partial D}{\partial y}|_{y=y_0}\neq 0$ in $(1.99513,2)\times (0,1/144).$
So, we conclude that the function $M$ has no critical point in $(0,2)\times (0,1)\times (0,1).$ 

\item Now, we study the interior of each of six face of the cuboid $T$.\\
\noindent \underline{On $p=0$}
\begin{equation}
n_1(x,y):=\dfrac{(1-x^2)(5 y^2 + x^2 y^2 + 6 x (1-y^2))}{2880},\label{7b2 c9.4}
\end{equation}
with $x,y\in (0,1)$. We note that, in $(0,1)\times(0,1)$, $n_1(x,y)$ does not have any critical point. As
\begin{equation*}
    \dfrac{\partial n_1}{\partial y}=\dfrac{(1-x^2)(x-1)(x-5)y}{1440}\neq 0\quad x,y\in (0,1).
\end{equation*}

\noindent \underline{On $p=2$}
\begin{equation}
    D(2,x,y):=\dfrac{1109}{6635520}\approx 0.000167131,\quad x,y\in (0,1).\label{7b2 c9.3}
\end{equation}
\noindent \underline{On $x=0$}
\begin{equation}
    n_2(p,y):=\dfrac{1109 p^6 +(4-p^2) (13632 p^3 y + 46080 (4-p^2) y^2 + 
 13824 p^2(1- y^2))}{42467328}\label{7b2 c9.1}
\end{equation}
with $p\in (0,2)$ and $y\in (0,1).$ We solve $\partial n_2/\partial p$ and $\partial n_2/\partial y$ to locate the points of maxima. On solving $\partial n_2/\partial y=0,$ we obtain
\begin{equation}
    y=\dfrac{71p^3}{48(13p^2-40)}(=:\epsilon).\label{7b2 cy}
\end{equation}
For the given range of $y$, we should have $\epsilon\in (0,1)$ which is possible only when $p>p_0:\approx 1.99513 \in (0,2)$ exists. On computations, $\partial n_2/\partial p=0$ gives
\begin{equation}
   18432 p - 9216 p^3 + 1109 p^5 + 27264 p^2 y - 11360 p^4 y - 
 141312 p y^2 + 39936 p^3 y^2=0.\label{7b2 c9}
\end{equation}
On substituting (\ref{7b2 cy}) in (\ref{7b2 c9}), we get,
\begin{equation}
    88473600 p - 101744640 p^3 + 38582784 p^5 - 5470944 p^7 + 169065 p^9=0.\label{7b2 c40}
\end{equation}
The solution of (\ref{7b2 c40}) in the interval $(0,2)$ is $p\approx 1.34821$, according to a numerical calculation. In $(0,2)\times(0,1)$, $n_2$ does not have a critical point.\\
\noindent \underline{On $x=1$}%, D(p,x,y)$ becomes
\begin{equation}
    n_3(p,y):=D(p,1,y)=\dfrac{398592p^2 - 121056 p^4 + 6461 p^6}{424673280}, \quad p\in (0,2).\label{7b2 c9.2}
\end{equation}
And when computing $\partial n_3/\partial p=0$, $p=:p_0\approx1.39681$ turns out to be the critical point. According to elementary calculations, $n_3$ reaches its  maximum value $\approx 0.000859123$ at $p_0.$\\
\noindent \underline{On $y=0$} 
\begin{align*}
    n_4(p,x):&=\dfrac{1}{424673280}\bigg(884736 x (1-x^2) + 
 2304 p^2 (24 - 192 x + 109 x^2 + 216 x^3 + 16 x^4)\\
    &\quad \quad\quad \quad\quad \quad\quad -96 p^4 (144 - 675 x + 880 x^2 + 720 x^3 + 192 x^4) \\
    &\quad \quad\quad \quad \quad\quad\quad+ 
 p^6 (1109 - 2376 x + 5424 x^2 + 2304 x^4) \bigg)=D(p,x,0).
\end{align*}
On computations,
\begin{align*}
    \dfrac{\partial n_4}{\partial x}&=\dfrac{(4-p^2)}{17694720}\bigg(9216 (1-3 x^2) - 48 p^2 (48- 109 x- 180 x^2 - 32 x^3)\\
     &\quad \quad\quad \quad \quad \quad+p^4 (99 -452 x - 384 x^3)\bigg)
\end{align*}
and \begin{align*}
    \dfrac{\partial n_4}{\partial p}&=\dfrac{1}{70778880}\bigg(768 p (24 - 192 x + 109 x^2 + 216 x^3 + 16 x^4)\\
    &\quad\quad\quad\quad\quad \quad-64 p^3 (144 - 675 x + 880 x^2 + 720 x^3 + 192 x^4)\\
    &\quad\quad\quad\quad\quad \quad+p^5 (1109 - 2376 x + 5424 x^2 + 2304 x^4)\bigg),
\end{align*}
we observe that only real solutions $(p,x)$ of the system of equations $\partial n_4/\partial x=0$ and $\partial n_4/\partial p=0$ are $(-2,-2.72495)$, $(-2.64507,-0.503718)$, $(-2.23927,-0.0103472)$, $(-2,-0.163482)$, $(2.23927,-0.0103472)$, $(-2.00568,-0.120347)$, $(2,-2.72495)$, $(-2,-0.0837889)$, $(0,-0.57735)$, $(2,-0.0837889)$, $(2.00568,-0.120347)$, $(2,-0.163482)$, $(2.64507,-0.503718)$ and\\ $(-0.91207,0.721981)$. Thus, no solution exists in $(0,2)\times (0,1)$, resulting in no critical points. 
 
\noindent \underline{On $y=1$}
\begin{align*}
n_5(p,x):&=\dfrac{1}{424673280}\bigg(36864 p x (3+4 x-3 x^2-4 x^3)+147456 (5- 4 x^2-x^4)\\
&\quad\quad \quad\quad\quad\quad- 2304 p^2 (160-24 x-261 x^2- 48 x^4)+ 
 768 p^3 (71 - 167 x^2 + 96 x^4) \\ 
&\quad\quad \quad\quad\quad\quad- 192 p^5 (71 + 36 x - 119 x^2 - 36 x^3 + 48 x^4)-96 p^4 (-480 + 45 x  \\
&\quad\quad\quad \quad\quad\quad + 1408 x^2 + 288 x^4) + 
 p^6 (1109 - 2376 x + 5424 x^2 + 2304 x^4)\bigg).
\end{align*}
On computations,
\begin{align*}
    \dfrac{\partial n_5}{\partial x}&=\dfrac{(4-p^2)}{17694720}\bigg(-6144 x (2 + x^2) + 384 p (3 + 8 x - 9 x^2 - 16 x^3)\\
     &\quad \quad\quad \quad \quad \quad+48 p^2 (12 + 197 x + 64 x^3)+ 16 p^3 (18 - 119 x - 54 x^2 + 96 x^3)\\
 &\quad \quad\quad \quad \quad \quad +p^4 (99 - 452 x -384 x^3)    \bigg)
\end{align*}
and \begin{align*}
    \dfrac{\partial n_5}{\partial p}&=\dfrac{1}{70778880}\bigg(6144 x (3+ 4 x - 3 x^2 - 4 x^3) + 
 768 p (-160 + 24 x + 261 x^2 + 48 x^4) \\
&\quad\quad\quad\quad\quad \quad-160 p^4 (71+ 36 x - 119 x^2 - 36 x^3 + 48 x^4) + 
 384 p^2 (71 - 167 x^2 + 96 x^4) \\
&\quad\quad\quad\quad\quad \quad+64 p^3 (480 -45 x- 1408 x^2 - 288 x^4) + 
 p^5 (1109 - 2376 x \\
&\quad\quad\quad\quad\quad \quad + 5424 x^2 + 2304 x^4)\bigg),
\end{align*}
we observe that only real solutions $(p,x)$ of the system of equations $\partial n_5/\partial x=0$ and $\partial n_5/\partial p=0$ are $(10.6873,0.22889)$, $(-2,-2.84595)$, $(2.70664,2.22228)$, $(2,2.03047)$, $(0,0)$, $(-2,-0.783981)$, $(2,-0.385354)$, $(-1.77349,0.0460014)$, $(2,-0.64512)$, $(1.98564,-0.516895)$, $(2.00568,-0.120347)$, $(2,0.657704)$, $(-2.33662,-1.5235)$, $(-1.03494,0.828224)$ and\\ $(1.36475,-1.03084)$. Thus, no solution exists in $(0,2)\times (0,1)$, resulting in no critical points.

\item Now, we calculate the maximum values achieved by $D(p,x,y)$ on the edges of the cuboid $T$.\\ From equation (\ref{7b2 c9.1}), we have $D(p,0,0)=h_1(p):=(55296 p^2 - 13824 p^4 + 1109 p^6)/424673280.$ It is easy to observe that $h_1'(p)=0$ for $p=\delta_0:=0$ and $p=1.83094$ in the interval $[0,2]$.\\ 
\begin{equation*}
    D(p,0,0)\leq 0.000169059.
\end{equation*}
Now considering (\ref{7b2 c9.1}) at $y=1,$ we get $D(p,0,1)=h_2(p):=(737280 - 368640 p^2 + 54528 p^3 + 46080 p^4 - 13632 p^5 + 1109 p^6)/424673280.$ It is easy to observe that $t_2'(p)$ is a decreasing function in $[0,2]$ and hence $p=0$ acts as its point of maxima. Thus
\begin{equation*}
    D(p,0,1)\leq \dfrac{1}{576}=0.00173611, \quad p\in [0,2].
    \end{equation*}
Through computations, (\ref{7b2 c9.1}) shows that $D(0,0,y)=y^2/576$, attains its maximum value at $y=1.$ This implies that
\begin{equation*}
    D(0,0,y)\leq \dfrac{1}{576}, \quad y\in [0,1].
\end{equation*}
Since, (\ref{7b2 c9.2}) is independent of $x$, we have $D(p,1,1)=D(p,1,0)=h_3(p):=(398592p^2 - 121056 p^4 + 6461 p^6)/424673280.$ Now, $h_3'(p)=294912 p - 175680 p^3 + 13206 p^5=0$ when $p=\delta_2:=0$ and $p=\delta_3:=1.39681$ in the interval $[0,2]$ with $\delta_2$ and $\delta_3$ as points of minima and maxima, respectively. Hence
\begin{equation*}
    D(p,1,1)=D(p,1,0)\leq 0.000859123,\quad p\in [0,2].
\end{equation*}
On substituting $p=0$ in (\ref{7b2 c9.2}), we get, $D(0,1,y)=0.$ The equation (\ref{7b2 c9.3}) is independent of the all the variables namely $p$, $x$ and $y$. Thus the maximum value of $D(p,x,y)$ on the edges $p=2, x=1; p=2, x=0; p=2, y=0$ and $p=2, y=1,$ respectively, is given by
\begin{equation*}
    D(2,1,y)=D(2,0,y)=D(2,x,0)=D(2,x,1)=\dfrac{1109}{6635520},\quad x,y\in [0,1].
\end{equation*}
From (\ref{7b2 c9.1}), we obtain $D(0,0,y)=y^2/576.$ A simple calculation shows that
\begin{equation*}
    D(0,0,y)\leq \dfrac{1}{576},\quad y\in [0,1].
    \end{equation*}
Using (\ref{7b2 c9.4}), we obtain $D(0,x,1)=h_4(x):=(5 - 4 x^2 - x^4)/2880.$ Upon calculations, we see that $h_4$ is a decreasing function in $[0,1]$ and hence attains its maximum value at $x=0.$ Thus
 \begin{equation*}
     D(0,x,1)\leq \dfrac{1}{576},\quad x\in [0,1].
 \end{equation*}
On again using (\ref{7b2 c9.4}), we get $D(0,x,0)=h_5(x):=x(1-x^2)/480.$ On further calculations, we get $t_5'(x)=0$ for $x=\delta_4:=1/\sqrt{3}.$ Also, $h_5(x)$ is an increasing function in $[0,\delta_4)$ and decreasing in $(\delta_4,1].$ So, it attains its maximum value at $\delta_4$. Thus
 \begin{equation*}
     D(0,x,0)\leq 0.000801875,\quad x\in [0,1].
 \end{equation*}
\end{enumerate}
In view of all the cases, the inequality (\ref{7b2 c9.5}) holds.
The function specified in (\ref{7b1 cextremal}) acts as an extremal function for the bounds of $|H_{2}(3)|$ having values $a_3=a_5=0$ and $a_4=1/24.$
\end{proof}

%On Acknowledgment %%%%%%%%
%		\section*{Acknowledgement}
	%		The author would like to thank the referees for the helpful
				%		suggestions.
				%		
				%%%% Bibliography  %%%%%%%%%%
\subsection*{Acknowledgment}
Neha Verma is thankful to the Department of Applied Mathematics, Delhi Technological University, New Delhi-110042 for providing Research Fellowship.


\begin{thebibliography}{99}
\bibitem{alarif} N. M. Alarifi, R. M. Ali\ and\ V. Ravichandran, On the second Hankel determinant for the $k$th-root transform of analytic functions, Filomat {\bf 31} (2017), no.~2, 227--245.
\bibitem{bulboca} D. Alimohammadi, E. Analouei Adegani, T. Bulboac{\u{a}}\ and\ N. E. Cho, Logarithmic coefficients for classes related to convex functions. Bull. Malays. Math. Sci. Soc., {\bf 44} (2021), no.~4, 2659--2673.
\bibitem{banga} S. Banga\ and\ S. S. Kumar, Sharp bounds of third Hankel determinant for a class of starlike functions and a subclass of $q$-starlike functions, Khayyam J. Math. {\bf 9} (2023), no.~1, 38--53.
\bibitem{sharp} S. Banga\ and\ S. S. Kumar, The sharp bounds of the second and third Hankel determinants for the class $\mathcal{SL}^*$, Math. Slovaca {\bf 70} (2020), no.~4, 849--862.
\bibitem{goodman vol1} A. W. Goodman, {\it Univalent functions}. Vol. I, Mariner Publishing Co., Inc., Tampa, FL, 1983.
\bibitem{kanas1} S. Kanas, Coefficient estimates in subclasses of the Carathéodory class related to conical domains. Acta Mathematica Universitatis Comenianae. New Series, {\bf 74} (2005), no.~2, 149-161.
\bibitem{4/9} B. Kowalczyk, A. Lecko\ and\ D. K. Thomas, The sharp bound of the third Hankel determinant for starlike functions, Forum Mathematicum. De Gruyter, (2022).
\bibitem{bellv} V. Kumar, N. E. Cho, V.  Ravichandran\ and\ H. M. Srivastava, Sharp coefficient bounds for starlike functions associated with the Bell numbers, Math. Slovaca {\bf 69} (2019), no.~5, 1053--1064.
\bibitem{nehastrip} S. S. Kumar\ and\ N. Verma, On a Subclass of Starlike Functions Associated with a Strip Domain. arXiv e-prints, pp.arXiv:2312.15266 (2023) (accepted in Ukrainian Mathematical Journal).
\bibitem{nehapetal} S. S. Kumar\ and\ N. Verma, Coefficient problems for starlike functions associated with a petal shaped domain. arXiv e-prints, pp.arxiv:2210.01435 (2022).
\bibitem{poojabean} S. S. Kumar\ and\ P. Yadav, On Starlike Functions Associated with a Bean Shaped Domain. arXiv e-prints, pp.arXiv:2403.14162 (2024).
\bibitem{lemma1} O. S. Kwon, A. Lecko\ and\ Y. J. Sim, On the fourth coefficient of functions in the Carath\'{e}odory class, Comput. Methods Funct. Theory {\bf 18} (2018), no.~2, 307--314.
\bibitem{rj} R. J. Libera\ and\ E. J. Z\l otkiewicz, Early coefficients of the inverse of a regular convex function, Proc. Amer. Math. Soc. {\bf 85} (1982), no.~2, 225--230.
\bibitem{ma-minda} W. C. Ma\ and\ D. Minda, A unified treatment of some special classes of univalent functions, Proc. Confer. Complex Anal. (Tianjin, 1992), 157--169.
\bibitem{bulboca2}K.  Marımuthu, J. Uma\ and\ T. Bulboac{\u{a}}, Coefficient estimates for starlike and convex functions associated with cosine function. Hacett. J. Math. Stat., {\bf 52} (2023), no.~3, 596-618.
\bibitem{pomi} C. Pommerenke, On the coefficients and Hankel determinants of univalent functions, J. London Math. Soc., {\bf 41} (1966), 111–-122.
\bibitem{poko} D. V. Prokhorov\ and\ J. Szynal, Inverse coefficients for $(\alpha,\beta )$-convex functions, Ann. Univ. Mariae Curie-Sk\l odowska Sect. A {\bf 35} (1981), 125--143 (1984).
\bibitem{shelly} V. Ravichandran\ and\ S. Verma, Bound for the fifth coefficient of certain starlike functions, C. R. Math. Acad. Sci. Paris {\bf 353} (2015), no.~6, 505--510.
%\bibitem{srivastava1} H. M. Srivastava, T. G. Shaba, M. Ibrahim, F. Tchier\ and\ B. Khan, Coefficient Bounds and Second Hankel Determinant for a Subclass of Symmetric Bi-Starlike Functions Involving Euler Polynomials. Bull. des Sciences Mathématiques, 2024, p.103405.
%\bibitem{srivastava3} H. M. Srivastava, G. Murugusundaramoorthy\ and\ T.Bulboac{\u{a}}, The second Hankel determinant for subclasses of bi-univalent functions associated with a nephroid domain. Revista de la Real Academia de Ciencias Exactas, Físicas y Naturales. Serie A. Matemáticas, {\bf 116} (2022), no.~4, p.145.
%\bibitem{srivastava2} H. M. Srivastava\ and\ A. K. Wanas, Applications of the Horadam polynomials involving $\lambda$-pseudo-starlike bi-univalent functions associated with a certain convolution operator. Filomat, {\bf 35} (2021), no.~14, 4645--4655.
\bibitem{nehacardioid} N. Verma\ and\ S. S. Kumar, A Conjecture on $H_3(1)$ for Certain Starlike Functions. Math. Slovaca, {\bf 73} (2023), no.~5, 1197--1206.
\bibitem{nehaexpo} N. Verma\ and\ S. S. Kumar, Certain Coefficient Problems of $\mathcal{S}^{*}_{e}$ and $\mathcal{C}_{e}$. arXiv e-prints, pp.arxive:2208.14644 (2022).










					
					
					
					
					
\end{thebibliography}
\end{document}